\documentclass[oneside,reqno,english]{amsart}
\usepackage[T1]{fontenc}
\usepackage[latin9]{inputenc}
\usepackage{geometry}
\usepackage{babel}
\usepackage{prettyref}
\usepackage{float}
\usepackage{mathrsfs}
\usepackage{enumitem}
\usepackage{amstext}
\usepackage{amsthm}
\usepackage{amssymb}
\usepackage{graphicx}
\usepackage[all]{xy}
\PassOptionsToPackage{normalem}{ulem}
\usepackage{ulem}
\usepackage[unicode=true,pdfusetitle,
 bookmarks=true,bookmarksnumbered=false,bookmarksopen=false,
 breaklinks=false,pdfborder={0 0 0},pdfborderstyle={},backref=false,colorlinks=false]
 {hyperref}
\hypersetup{
 urlcolor=black,colorlinks=true,citecolor=black,linkcolor=black,linktocpage=true,pdfstartview={XYZ null null 1.00},pdfpagemode=UseOutlines}

\makeatletter

\providecommand{\tabularnewline}{\\}

\numberwithin{equation}{section}
\numberwithin{figure}{section}
\numberwithin{table}{section}
\theoremstyle{plain}
\newtheorem{thm}{\protect\theoremname}[section]
\theoremstyle{definition}
\newtheorem{defn}[thm]{\protect\definitionname}
\theoremstyle{plain}
\newtheorem{lem}[thm]{\protect\lemmaname}
\theoremstyle{plain}
\newtheorem{prop}[thm]{\protect\propositionname}
\theoremstyle{plain}
\newtheorem{cor}[thm]{\protect\corollaryname}
\theoremstyle{plain}
\newtheorem*{question*}{\protect\questionname}
\theoremstyle{remark}
\newtheorem{rem}[thm]{\protect\remarkname}
\theoremstyle{definition}
\newtheorem{example}[thm]{\protect\examplename}
\theoremstyle{remark}
\newtheorem*{notation*}{\protect\notationname}
\newcommand\thmsname{\protect\theoremname}
\newcommand\nm@thmtype{theorem}
\theoremstyle{plain}

\newenvironment{namedthm}[1][Undefined Theorem Name]{
  \ifx{#1}{Undefined Theorem Name}\renewcommand\nm@thmtype{theorem*}
  \else\renewcommand\thmsname{#1}\renewcommand\nm@thmtype{namedtheorem}
  \fi
  \begin{\nm@thmtype}}
  {\end{\nm@thmtype}}
\theoremstyle{plain}
\newtheorem{question}[thm]{\protect\questionname}
\theoremstyle{remark}
\newtheorem*{acknowledgement*}{\protect\acknowledgementname}

\allowdisplaybreaks
\usepackage{needspace}
\usepackage{prettyref}
\usepackage{enumitem}
\usepackage{bbm}

\newrefformat{prop}{Proposition~\ref{#1}}
\newrefformat{cor}{Corollary~\ref{#1}}
\newrefformat{def}{Definition~\ref{#1}}
\newrefformat{rem}{Remark~\ref{#1}}
\newrefformat{thm}{Theorem~\ref{#1}}
\newrefformat{lem}{Lemma~\ref{#1}}
\newrefformat{fig}{Figure~\ref{#1}}
\newrefformat{exa}{Example~\ref{#1}}

\setlist[enumerate,1]{label=\textup{(}\roman*\textup{)},ref=\roman*}
\setlist[enumerate,2]{label=\textup{(}\alph*\textup{)},ref=\theenumi \alph*}

\makeatletter
\def\paragraph*{\@startparagraph{paragraph}{3}%
  \z@{.5\linespacing\@plus.7\linespacing}{.1\linespacing}%
  {\normalfont\itshape}}
\makeatother

\makeatletter
\def\paragraph{\@startsection{paragraph}{4}%
   \z@{.5\linespacing\@plus.7\linespacing}{.1\linespacing}%
  {\normalfont\bfseries}}
\makeatother

\AtBeginDocument{
  
}

\makeatother

\providecommand{\acknowledgementname}{Acknowledgement}
\providecommand{\corollaryname}{Corollary}
\providecommand{\definitionname}{Definition}
\providecommand{\examplename}{Example}
\providecommand{\lemmaname}{Lemma}
\providecommand{\notationname}{Notation}
\providecommand{\propositionname}{Proposition}
\providecommand{\questionname}{Question}
\providecommand{\remarkname}{Remark}
\providecommand{\theoremname}{Theorem}

\begin{document}
\title[]{Superposition, reduction of multivariable problems, and approximation}
\author{Palle Jorgensen and Feng Tian}
\address{(Palle E.T. Jorgensen) Department of Mathematics, The University of
Iowa, Iowa City, IA 52242-1419, U.S.A. }
\email{palle-jorgensen@uiowa.edu}
\urladdr{http://www.math.uiowa.edu/\textasciitilde jorgen/}
\address{(Feng Tian) Department of Mathematics, Hampton University, Hampton,
VA 23668, U.S.A.}
\email{feng.tian@hamptonu.edu}
\begin{abstract}
We study reduction schemes for functions of \textquotedblleft many\textquotedblright{}
variables into system of functions in one variable. Our setting includes
infinite-dimensions. Following Cybenko-Kolmogorov, the outline for
our results is as follows: We present explicit reductions schemes
for multivariable problems, covering both a finite, and an infinite,
number of variables. Starting with functions in \textquotedblleft many\textquotedblright{}
variables, we offer constructive reductions into superposition, with
component terms, that make use of only functions in one variable,
and specified choices of coordinate directions. Our proofs are transform
based, using explicit transforms, Fourier and Radon; as well as multivariable
Shannon interpolation.
\end{abstract}

\maketitle
\tableofcontents{}

\section{Introduction}

In this paper we consider a general problem, which deals with functions
in \textquotedblleft many\textquotedblright{} variables, and their
possible reduction into superposition, with component terms that make
use of only functions in one variable, and suitable choices of coordinate
directions. The problem reads as follows, in brief summary:

\emph{Reduction of functions of ``many'' variables into system of
functions in one variable.}\\

Classically, variants of the question were first asked in the case
of functions of a finite number of variables, say $m$ (\textquotedblleft large\textquotedblright );
see Theorems \ref{thm:B1} and \ref{thm:B2} below. If $F$ is a function
on a subset in $\mathbb{R}^{m}$, it is natural to ask that $F$ allow
a reconstruction, or approximation, via choices of a suitable set
of coordinate directions, each such direction given by a non-zero
vector $w$ in $\mathbb{R}^{m}$. When a system $W$ of directions
is specified, one wishes to approximate $F$ with an associated system
of functions (of one variable), one for each direction specified by
the set $W$. Following Kolmogorov, one says that $F$ admits a superposition;
see \prettyref{thm:B2}. Here we shall also be concerned with functions
in an infinite number of variables, especially functions $F$ which
arise as random variables in some specified probability space; see
\prettyref{prop:A3}, and \prettyref{fig:Fdim1} below. In this case,
it is natural to think of \textquotedblleft directions\textquotedblright{}
as a choice of real valued random variables, one for each direction.

The Universal Approximation Theorem (UAT) as developed by Kolmogorov
and Cybenko (see \prettyref{thm:B1}) is of current interest as it
provides a partial explanation for why neural networks are able to
\textquotedblleft learn'' from data. However Cybenko's variant of
UAT, dealing with sigmoid as activation function, is more existential
than constructive. We attempt here to remedy that somewhat: We aim
to quantify defect, meaning the lack of density; hence a variety of
choices of UAT-activation functions.

\textbf{Organization.} We first outline our infinite-dimensional setting:
A choice of our probabilistic framework, including specification of
the appropriate probability space, and our choice of systems of random
variables. In sect \ref{sec:ck}, we expand on our extension of, and
approach to, a transform-setting for generalized Universal Approximation
Theorems (UAT), and Cybenko-Kolmogorov. For this purpose, we introduce,
in sect \ref{sec:ps}, a new projective space of equivalence classes.
In sections \ref{sec:mf}--\ref{sec:fr}, we state our results, and
transform based algorithms. This includes the transforms of Fourier
and Radon; as well as a multivariable Shannon interpolation adapted
to UAT.

\subsection{Infinite dimensions and a probabilistic framework}
\begin{defn}
Let $\left(\Omega,\mathscr{F},\mathbb{P}\right)$ be a probability
space, and 
\begin{equation}
\mathbb{E}\left(\cdot\cdot\right)=\int_{\Omega}\left(\cdot\cdot\right)d\mathbb{P},\label{eq:A1}
\end{equation}
be the mean (or expectation). Let 
\begin{equation}
X:\Omega\longrightarrow\mathbb{R}\label{eq:A2}
\end{equation}
be measurable with respect to $\mathscr{F}$ on $\Omega$, and $\mathscr{B}$
(Borel $\sigma$-algebra) on $\mathbb{R}$; we say that $X$ is a
\emph{random variable}.

In the finite dimensional case, $m<\infty$, we shall consider 
\[
\Omega=J_{m}=\left[-1,1\right]^{m}=\left\{ x=\left(x_{j}\right)_{1}^{m}\mid-1\leq x_{j}\leq1\right\} ,
\]
and $L^{2}\left(J_{m}\right)$ with the standard Lebesgue measure.
The measure can be normalized, so that 
\[
\lambda_{m}:=\frac{1}{2^{m}}dx=\frac{1}{2^{m}}dx_{1}\cdots dx_{m}
\]
satisfies $\lambda_{m}\left(J^{m}\right)=1$. See Figures \ref{fig:Fdim1}
and \ref{fig:Fdim2}.
\end{defn}

\paragraph*{The infinite dimensional case}
\begin{lem}
Let $\left(\Omega,\mathscr{F}\right)$ be a measure space, and let
$X:\Omega\rightarrow\mathbb{R}$ be measurable, where $\mathbb{R}$
is equipped with the Borel $\sigma$-algebra $\mathscr{B}$. Then,
if $F$ is measurable $\left(\Omega,\mathscr{F}\right)\rightarrow\left(\mathbb{R},\mathscr{B}\right)$,
TFAE:
\begin{enumerate}
\item \label{enu:lemA2-1}$\exists$ $\varphi:\mathbb{R}\rightarrow\mathbb{R}$,
$\mathscr{B}$-measurable s.t. $F=\varphi\circ X$, and 
\item \label{enu:lemA2-2}$F$ is measurable w.r.t. the pullback $\sigma$-algebra
$\mathscr{F}_{X}:=X^{-1}\left(\mathscr{B}\right)$; see \prettyref{fig:A3}.
\end{enumerate}
\end{lem}

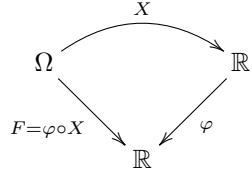
\begin{figure}[H]
\[
\xymatrix{\Omega\ar@/^{1.3pc}/[rr]^{X}\ar[dr]_{F=\varphi\circ X} &  & \mathbb{R}\ar[dl]^{\varphi}\\
 & \mathbb{R}
}
\]

\caption{\label{fig:A3}Measurable.}
\end{figure}

\begin{proof}
The implication \eqref{enu:lemA2-1}$\Longrightarrow$\eqref{enu:lemA2-2}
is immediate from the definitions. Note that if $J\in\mathscr{B}$,
then 
\begin{align}
\left(\varphi\circ X\right)^{-1}\left(J\right) & =X^{-1}\left(\varphi^{-1}\left(J\right)\right);\;\text{and}\\
\chi_{_{J}}\circ X & =\chi_{_{X^{-1}\left(J\right)}},
\end{align}
holds for the respective indicator functions. 

Hence if $\varphi:\mathbb{R}\rightarrow\mathbb{R}$ is a simple function,
$J_{i}\in\mathscr{B}$, $c_{i}\in\mathbb{R}$, $1\leq i\leq N$, 
\begin{equation}
\varphi=\sum_{i}c_{i}\chi_{_{J_{i}}},
\end{equation}
we get 
\[
\varphi\circ X=\sum_{i}c_{i}\chi_{_{X^{-1}\left(J_{i}\right)}}.
\]
Since measurability is characterized via approximation with the respective
simple functions, the remaining implication \eqref{enu:lemA2-2}$\Longrightarrow$\eqref{enu:lemA2-1}
now follows. 
\end{proof}
\begin{prop}
\label{prop:A3}If $\mu_{X}:=\mathbb{P}\circ X^{-1}$ denotes the
\uline{distribution} of $X$ in \eqref{eq:A2}, then 
\begin{equation}
L^{2}\left(\mathbb{R},\mu_{X}\right)\ni\varphi\xrightarrow{\quad T_{X}\quad}\varphi\circ X\in L^{2}\left(\Omega,\mathbb{P}\right)
\end{equation}
is \uline{isometric}; and the \emph{adjoint} operator 
\[
T_{X}^{*}:L^{2}\left(\Omega,\mathbb{P}\right)\longrightarrow L^{2}\left(\mathbb{R},\mu_{X}\right)
\]
is \uline{coisometric}. It is given by the $X$-conditional expectation
\begin{equation}
T_{X}^{*}\left(F\right)\left(x\right)=\mathbb{E}_{X=x}\left(F\mid\mathscr{F}_{X}\right),\quad F\in L^{2}\left(\Omega,\mathbb{P}\right).\label{eq:A4}
\end{equation}
\end{prop}

\begin{figure}
\includegraphics[width=0.6\textwidth]{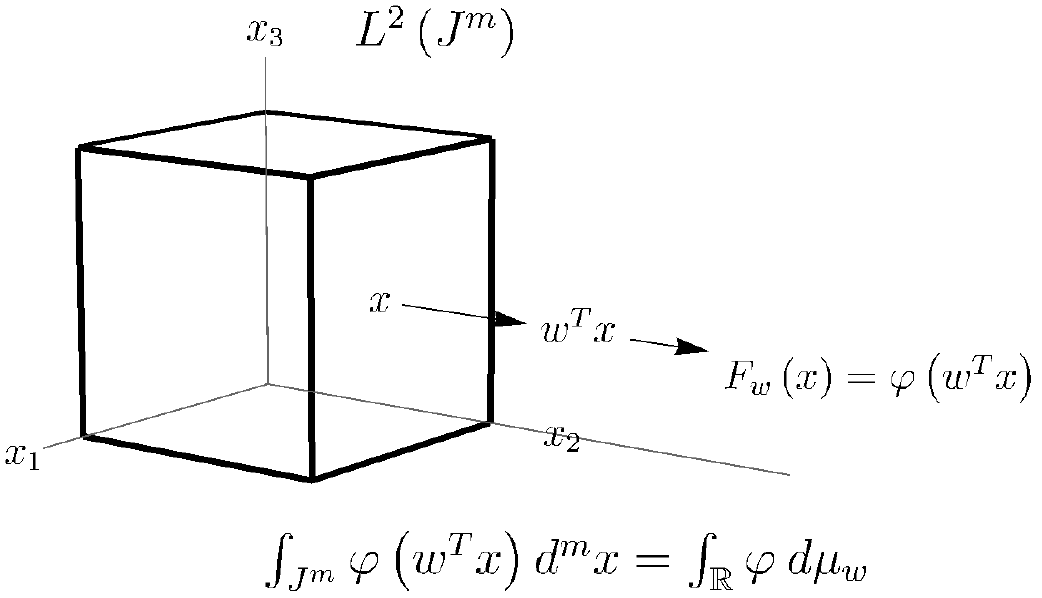}

\caption{\label{fig:Fdim1}Finite dimensions}
\end{figure}

\begin{figure}
\includegraphics[width=0.5\textwidth]{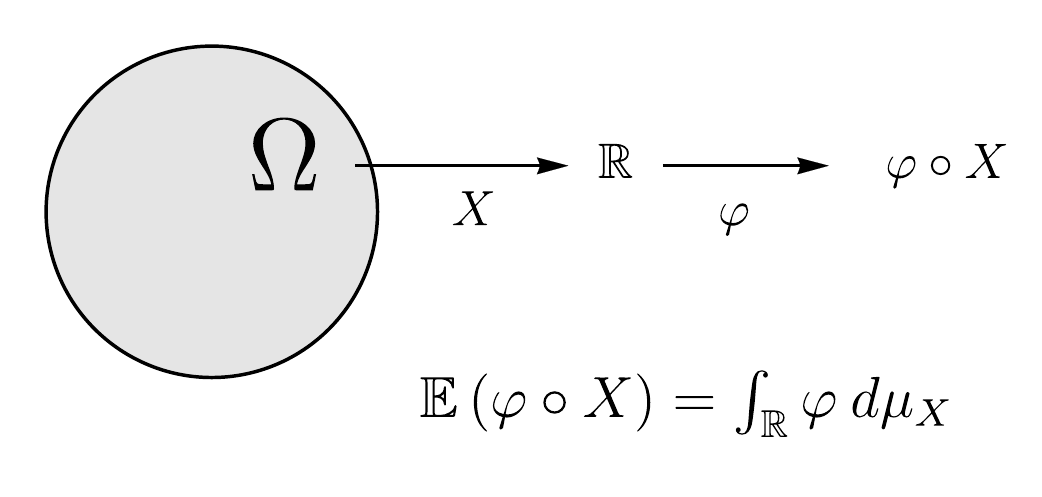}

\caption{\label{fig:Fdim2}Infinite-dimensional probability space $\left(\Omega,\mathscr{F},\mathbb{P}\right)$}
\end{figure}

\begin{proof}
The meaning of the RHS in \eqref{eq:A4} is as follows. It is a double
conditioning: (i) Conditional expectation by the sub $\sigma$-algebra
$\mathscr{F}_{X}=X^{-1}\left(\mathscr{B}\right)$ in $\mathscr{F}$
generated by the random variable $X$ (i.e., the pullback of the Borel
sets under $X$); and (ii) secondly we condition by the initial condition
$X=x$; so $T_{X}^{*}\left(F\right)$ becomes a function defined on
$\mathbb{R}$. Said differently, $T_{X}^{*}\left(F\right)=\varphi$
is the function satisfying, 
\begin{equation}
\mathbb{E}\left(F\mid\mathscr{F}_{X}\right)=\varphi\circ X.
\end{equation}
To see this, recall that the conditional expectation $\mathbb{E}\left(\cdot\cdot\mid\mathscr{F}_{X}\right)$
is the orthogonal projection in $L^{2}\left(\Omega,\mathscr{F},\mathbb{P}\right)$
onto the subspace generated by the functions $\psi\circ X$, as $\psi$
varies over all the Borel measurable functions on $\mathbb{R}$. Moreover,
the conditional expectation satisfies 
\begin{equation}
\mathbb{E}\left(\left(\psi\circ X\right)\mathbb{E}\left(F\mid\mathscr{F}_{X}\right)\right)=\mathbb{E}\left(\left(\psi\circ X\right)F\right)\label{eq:A6}
\end{equation}
valid for all $\psi$, and all $F\in L^{2}\left(\Omega,\mathscr{F},\mathbb{P}\right)$.

It follows that the two operators $T_{X}$, and adjoint $T_{X}^{*}$
satisfy the following identities:
\begin{align*}
T_{X}^{*}T_{X} & =I_{L^{2}\left(\mathbb{R},\,\mu_{X}\right)},\;\text{and}\\
T_{X}T_{X}^{*} & =\mathbb{E}\left(\cdot\cdot\mid\mathscr{F}_{X}\right).
\end{align*}
\end{proof}

\begin{cor}
Let $\left(\Omega,\mathscr{F},\mathbb{P}\right)$ denote a probability
space. Let $\mathscr{W}$ be a system of real valued random variables,
and assumed contained in $L^{2}\left(\Omega,\mathscr{F},\mathbb{P}\right)$.
Let $F$ be a random variable with $\mathbb{E}\left(\left|F\right|^{2}\right)<\infty$;
then TFAE:
\begin{enumerate}
\item $F\in L^{2}\left(\Omega,\mathbb{P}\right)\ominus\left\{ \varphi\circ X\mathrel{;}\varphi\in\mathscr{C},\:X\in\mathscr{W}\right\} $,
where $\mathscr{C}:=C_{b}\left(\mathbb{R},\mathbb{R}\right)$ or $C_{b}\left(\mathbb{R},\mathbb{C}\right)$;
\item $T_{X}^{*}\left(F\right)=0$, $\forall X\in\mathscr{W}$.
\end{enumerate}
\end{cor}

\begin{proof}
The result is immediate from \eqref{eq:A4} and the following: 
\[
\int_{\mathbb{R}}T_{X}^{*}\left(G\right)\varphi\,d\mu_{X}=\mathbb{E}\left(G\left(\varphi\circ X\right)\right),
\]
valid for $\forall G\in L^{2}\left(\Omega,\mathbb{P}\right)$, and
all $\varphi\in\mathscr{C}$. 
\end{proof}
\begin{question*}
Given $F\in L^{2}\left(J^{m}\right)$, how do we get representations
of $F$ in terms of coordinate functions $J^{m}\ni x\rightarrow w\cdot x\in\mathbb{R}$?
Here, $w\cdot x=w^{T}x=w_{1}x_{1}+\cdots+w_{m}x_{m}$.
\end{question*}
If $L\in\mathscr{B}$, a Borel set in $\mathbb{R}$, then the respective
measures $\mu_{w}$ and $\mu_{X}$ from Figures \ref{fig:Fdim1} \&
\ref{fig:Fdim2} are as follows (distributions of random variables):
\begin{align}
\mu_{w}\left(L\right) & =\lambda_{m}\left(\left\{ x\in J^{m}\mid X_{w}\left(x\right):=w^{T}x\in L\right\} \right),\\
\mu_{X}\left(L\right) & =\mathbb{P}\left(\left\{ w\in\Omega\mid X\left(w\right)\in L\right\} \right).\label{eq:A8}
\end{align}
Equivalently, $\mu_{w}=\lambda_{m}\circ X_{w}^{-1}$, and $\mu_{X}=\mathbb{P}\circ X^{-1}$;
see also \eqref{eq:A1} \& \eqref{eq:A2}.

Let $\mathscr{C}:=C_{b}\left(\mathbb{R},\mathbb{R}\right)$ or $C_{b}\left(\mathbb{R},\mathbb{C}\right)$.
For $\varphi\in\mathscr{C}$, we consider $x\longrightarrow\varphi\left(w^{T}x\right)$
as a function on $J^{m}$. See Figures \ref{fig:Fdim1}--\ref{fig:Fdim2}.
And we define the following subspace
\begin{equation}
\mathscr{H}\left(w\right)=\overline{\left\{ \varphi\left(w^{T}x\right)\mathrel{;}\varphi\in\mathscr{C}\right\} }^{L^{2}\left(J^{m}\right)}.\label{eq:A9}
\end{equation}
(The overbar means closure, and the superscript refers to the norm.)
\begin{question*}
What are minimal subsets $W$ of vectors $w\in\mathbb{R}^{m}\backslash\left(0\right)$
such that the closed span of $\left\{ \mathscr{H}\left(w\right)\right\} _{w\in W}$
is equal to $L^{2}\left(J^{m}\right)$?
\end{question*}
For additional details regarding probability spaces, random variables
and distributions, see e.g., \cite{MR3798397,MR3799634,MR3803594,MR3830636}.
Also see \cite{MR2172247,MR2376788,MR3025990,MR2558696} for approximation
and Kolmogorov's superposition theorem.
\begin{rem}
Our present results are partially motivated by the following question:
\textquotedblleft How do neural nets (NN) learn?\textquotedblright{}
For example, consider a single (hidden) layer neural net. With a specified
starting point given as an arbitrary $w$ (could be set to zero);
how will it then be updated through \textquotedblleft backward propagation\textquotedblright ,
or \textquotedblleft backprop\textquotedblright , as each training
example passes through the system? Note that the process aims for
the output to then be able to approximate uniformly any continuous
function, i.e., (in the language of learning machines) the system
learns a hypothesis. One thing that makes this successful for neural
network (NN) is that backprop is fast enough to be able to train very
large networks.

However, if $w$ is fixed (none zero), and we modify $\varphi$ as
in \eqref{eq:A9}, then there are several related questions:
\begin{enumerate}
\item How to choose the initial $\varphi$?

In standard NN-models, $w$ is fixed from the outset as activation
function. A popular choice in recent literature is the rectified linear
unit (ReLU); see e.g., \cite{Nrelu}. Other options include sigmoid,
ArcTan, etc.
\item As training examples pass through the system, how is $\varphi$ updated?
One would expect a variant of backprop.
\item Is this version of backprop, if exists, fast enough to handle large
networks?
\end{enumerate}
In view of (ii), getting the optimal $\varphi$ is harder than training
$w$ in standard NNs. In a standard NN (assuming single layer), $\varphi$
is fixed, and $w$ is updated recursively. Note that $w$ lives in
a finite dimensional space.

If $w$ is fixed instead (see \eqref{eq:A9}), training $\varphi$
means one seeks the optimal solution in the set of \emph{all} continuous
functions, or differentiable functions, satisfying certain conditions
(see \prettyref{sec:ck} below). The parameter $\varphi$, now a function,
must be modified slightly at each iteration. The ambient space for
$\varphi$ is infinite-dimensional.
\end{rem}

\begin{example}
In the special case, when $\varphi$ is fixed, we have a standard
NN-model which serves to \textquotedblleft train\textquotedblright{}
$w$. See \prettyref{fig:nn-1} below. Our present setup is much more
subtle, and it is hard to compare this to the case when $w$ is fixed,
and $\varphi$ varies. Our present results serve to motivate a number
of algorithmic approaches to the \textquotedblleft hard\textquotedblright{}
case.

In \prettyref{fig:nn-1}, we illustrate the structure of a standard
$L$-layer neural net for multiclass classification. (For example,
in classical handwritten digit recognition, the number of classes
is 10. See \prettyref{fig:nn-2}.) The parameters are specified as
follows:
\begin{itemize}
\item $x\in M_{m,n}$, $n$ training examples, each of dimension $m$;
\item $y\in M_{k,n}$, $k$ classes, where $y\left(\cdot,j\right)$ is the
standard basis in $\mathbb{R}^{k}$. 
\item $w_{j}\in M_{s_{j},s_{j-1}}$; $s_{j}=$ number of units in layer
$j$, $s_{0}=m$.
\item $b_{j}\in\mathbb{R}^{s_{j}}$, ``bias''.
\item $g_{j}:M_{s_{j},n}\rightarrow M_{s_{j},n}$, $g_{j}\left(z\right)=\left(g_{j}\left(z_{s,t}\right)\right)$,
for $z=\left(z_{st}\right)\in M_{s_{j},n}$.
\end{itemize}
During the training process, the system iterates to approximate the
optimal parameters:
\begin{itemize}
\item Forward propagation;
\item Calculation of cost;
\item Backward propagation to obtain $dw$, $db$;
\item Update parameters, $w=w-\beta dw$, $b=b-\beta db$, where $\beta>0$
denotes the learning rate.
\end{itemize}
\end{example}

\begin{figure}
$x\longrightarrow g_{1}\left(w_{1}\left(\cdot\right)+b_{1}\right)\longrightarrow g_{2}\left(w_{2}\left(\cdot\right)+b_{2}\right)\longrightarrow\cdots\longrightarrow g_{L}\left(w_{L}\left(\cdot\right)+b_{L}\right)$

\caption{\label{fig:nn-1}basic structure of a standard $L$-layer neural net}
\end{figure}

\begin{figure}
\includegraphics[width=0.2\textwidth]{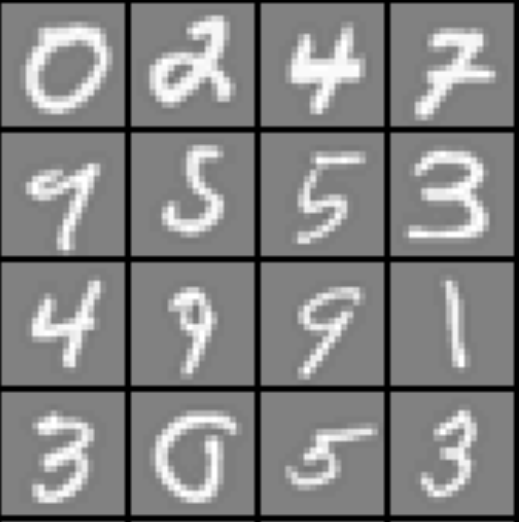}

\caption{\label{fig:nn-2}Example: handwritten digit recognition.}
\end{figure}

\section{\label{sec:ck}Theorems by Cybenko and Kolmogorov}

Our present investigations are motivated in part by the following
Universal Approximation Theorem by Cybenko: 
\begin{thm}[Cybenko \cite{MR1015670,MR1178852}]
\label{thm:B1}Let $\varphi:\mathbb{R}\rightarrow\mathbb{R}$ be
given, satisfying the following two conditions: 
\begin{enumerate}
\item \label{enu:B1-1}$\varphi$ is continuous, and
\item \label{enu:B1-2}the following two limits exist:
\begin{enumerate}
\item $\lim_{t\rightarrow-\infty}\varphi\left(t\right)=0$, and 
\item $\lim_{t\rightarrow+\infty}\varphi\left(t\right)=1$. 
\end{enumerate}
\end{enumerate}
Then the span of the double-indexed functions 
\begin{equation}
\left\{ \varphi\left(w^{T}x+\lambda\right)\right\} _{w\in\mathbb{R}^{m}\backslash\left(0\right),\,\lambda\in\mathbb{R}}\label{eq:B1}
\end{equation}
is uniformly dense in $C\left(J^{m}\right)$.

\end{thm}

\begin{proof}[Proof sketch]
 The required reasoning is three fold. Specifically, the proof establishes
the following three assertions:
\begin{enumerate}[label=(\arabic{enumi})]
\item The family of functions in \eqref{eq:B1} is dense in $C\left(J^{m}\right)$
$\Longleftrightarrow$ The integral of the function system \eqref{eq:B1}
being zero for Borel measures $\mu$ on $\mathbb{R}$ of finite total
variation, only when $\mu$ is zero.
\item Subcases: The integral with respect to $\mu$ is zero iff $\mu$ vanishes
on all the half-planes, $\Pi_{w,s}$, $\Pi_{w,s}^{+}$ and $\Pi_{w,s}^{-}$;
see details below.
\item The half-planes generate the $\sigma$-algebra on $J^{m}$.
\end{enumerate}
Since most of the ideas going into the proof may be found in the papers
\cite{MR1015670,MR1178852}, we shall be brief: We provide the following
sketch for the benefit of the reader.

Let $\varphi$ be as specified, and let $\mu$ be a Borel measure
on $\mathbb{R}$ of finite total variation. For $\lambda,s,t\in\mathbb{R}$,
and $w\in\mathbb{R}^{m}\backslash\left(0\right)$, consider the following
family of functions on $\mathbb{R}^{m}$:
\begin{equation}
\varphi\left(\lambda\left(w^{T}x+s\right)+t\right),\quad x\in J^{m};
\end{equation}
and set 
\begin{align*}
\Pi_{w,s} & =\left\{ x\in\mathbb{R}^{m}\mathrel{;}w^{T}x+s=0\right\} ,\\
H_{w,s}^{+} & =\left\{ x\in\mathbb{R}^{m}\mathrel{;}w^{T}x+s>0\right\} ,
\end{align*}
and 
\[
H_{w,s}^{-}=\left\{ x\in\mathbb{R}^{m}\mathrel{;}w^{T}x+s<0\right\} .
\]
Then 
\begin{align*}
 & \lim_{\lambda\rightarrow+\infty}\int_{J^{m}}\varphi\left(\lambda\left(w^{T}x+s\right)+t\right)d\mu\left(x\right)\\
= & \mu\left(H_{w,s}^{+}\right)+\varphi\left(t\right)\mu\left(\Pi_{w,s}\right);
\end{align*}
and 
\begin{align*}
 & \lim_{\lambda\rightarrow-\infty}\int_{J^{m}}\varphi\left(\lambda\left(w^{T}x+s\right)+t\right)d\mu\left(x\right)\\
= & \mu\left(H_{w,s}^{-}\right)+\varphi\left(t\right)\mu\left(\Pi_{w,s}\right).
\end{align*}
\end{proof}
For additional details regarding the approximation problems, see e.g.,
\cite{MR3816236,MR3816652,MR3836380,MR679449,MR740851,MR2817068}.

A second motivation for our present considerations is Kolmogorov\textquoteright s
superposition theorem. The latter in turn is Kolmogorov's reply to
Hilbert's 13th Problem.
\begin{thm}[Kolmogorov, see \cite{MR2558696}]
\label{thm:B2}Let $f:\left[0,1\right]^{n}\rightarrow\mathbb{R}$
be an arbitrary multivariable continuous function. Then it has the
representation
\[
f\left(x_{1},\dots,x_{n}\right)=\sum_{q=0}^{2n}\Phi_{q}\left(\sum_{p=1}^{n}\psi_{q,p}\left(x_{p}\right)\right),
\]
with continuous one-dimensional outer and inner functions $\Phi_{q}$
and $\psi_{q,p}$. All these functions $\Phi_{q}$ and $\psi_{q,p}$
are defined on the real line. The inner functions $\psi_{q,p}$ are
independent of the function $f$. 
\end{thm}

Hilbert originally posed his 13th problem for algebraic functions
(Hilbert 1927, ``...Existenz von algebraischen Funktionen...'',
i.e., ``...existence of algebraic functions...''.) Hilbert asked
for a process whereby a function of several variables may possibly
be constructed using only functions of two variables. Hilbert's conjecture,
that it is not always possible to find such a solution, was disproven
in 1957.

However, there also was a later version of the problem where Hilbert
asked instead whether there are solution in the class of continuous
functions. It is the second version of Hilbert's 13th problem that
concerns us here, by way of motivation. More specifically, a generalization
of the second (``continuous'') variant of Hilbert's problem is the
question: Can every continuous function of three variables be expressed
as a composition of finitely many continuous functions of two variables?
The affirmative answer to this general question was given in 1957
by Vladimir Arnold. Earlier Kolmogorov had shown that any function
of several variables can be constructed with a finite number of three-variable
functions. Arnold then expanded on this work to show that only two-variable
functions were in fact required, thus answering Hilbert's question
in the context of continuous functions. In our present consideration,
we shall consider versions of the question for functions on the hyper
cube $J^{m}$, where $m$ is finite, but \textquotedblleft large.\textquotedblright{}
For the multivariable functions, we shall consider both the continuous
case, as well as the variant for the problem in the Hilbert space
$L^{2}\left(J^{m}\right)$.

\subsection{An infinite-dimensional analogue of Cybenko's theorem}

Let $\dot{\mathbb{R}}$ denote the one-point compactification of $\mathbb{R}$,
and consider the infinite product space 
\begin{equation}
\Omega:=\dot{\mathbb{R}}^{\mathbb{N}}=\dot{\mathbb{R}}\times\dot{\mathbb{R}}\times\dot{\mathbb{R}}\times\cdots.\label{eq:B3}
\end{equation}

On $\Omega$, we shall consider the topology generated by the \emph{cylinder
sets}; so that $\Omega$ is compact, by the Tychonov-theorem. 

The $\sigma$-algebra generated by the cylinder-sets will be denoted
$\mathscr{F}$; and we shall be considering probability spaces $\left(\Omega,\mathscr{F},\mathbb{P}\right)$. 

A system $\mathscr{W}$ of random variables $X:\Omega\rightarrow\mathbb{R}$
is said to be \emph{separating} iff (Def) the following subsets of
$\Omega$ generate $\mathscr{F}$: 
\begin{equation}
\begin{split}\left\{ X\left(\cdot\right)>s\right\}  & =X^{-1}\left((s,\infty]\right),\;\text{and}\\
\left\{ X\left(\cdot\right)=s\right\}  & =X^{-1}\left(\left\{ s\right\} \right),
\end{split}
\label{eq:B4}
\end{equation}
where $s\in\mathbb{R}$, and $X\in\mathscr{W}$ are arbitrary. More
precisely, we require that, if a signed measure $\mu$ on $\left(\Omega,\mathscr{F}\right)$
of finite total variation vanishes on the sets in \eqref{eq:B4},
then $\mu$ must be zero. 
\begin{thm}
Let $\left(\Omega,\mathscr{F},\mathbb{P}\right)$ be as above; let
$\mathscr{W}$ be a separating system of continuous random variables;
and let $\varphi:\mathbb{R}\rightarrow\mathbb{R}$ be a fixed function
satisfying conditions \eqref{enu:B1-1} \& \eqref{enu:B1-2} in \prettyref{thm:B1}.
Then the span of the double-indexed system of functions (on $\Omega$),
\begin{equation}
\left\{ \varphi\left(X\left(\cdot\right)+\lambda\right)\right\} _{X\in\mathscr{W},\,\lambda\in\mathbb{R}}
\end{equation}
is uniformly dense in $C\left(\Omega\right)$; i.e., dense w.r.t.
the $\left\Vert \cdot\right\Vert _{\infty}$-norm on $C\left(\Omega\right)$. 
\end{thm}

\begin{proof}
As the present arguments are close to those outlined in the proof
of \prettyref{thm:B1} above, we shall only sketch the details. 

First consider the following system of functions on $\Omega$
\begin{equation}
\varphi\left(\lambda\left(X\left(\cdot\right)+s\right)+t\right)
\end{equation}
where $\varphi$ is as stated, i.e., satisfying conditions \eqref{enu:B1-1}
\& \eqref{enu:B1-2} in \prettyref{thm:B1}; and where $X\in\mathscr{W}$,
and $\lambda,s,t\in\mathbb{R}$. 

Let $\mu$ be a signed measure on $\left(\Omega,\mathscr{F}\right)$
of finite total variation. We now show that if 
\[
\int_{\Omega}\varphi\left(\lambda\left(X\left(\cdot\right)+s\right)+t\right)d\mu=0
\]
for all $X\in\mathscr{W}$, and $\lambda,s,t\in\mathbb{R}$, then
$\mu=0$. But 
\begin{alignat*}{2}
 &  &  & \lim_{\lambda\rightarrow+\infty}\int_{\Omega}\varphi\left(\lambda\left(X\left(\cdot\right)+s\right)+t\right)d\mu\\
 & = & \: & \mu\left(\left\{ X\left(\cdot\right)+s>0\right\} \right)+\varphi\left(t\right)\mu\left(\left\{ X\left(\cdot\right)+s=0\right\} \right),
\end{alignat*}
and the desired conclusion now follows precisely as in Cybenko's reasoning.
Recall that the family $\mathscr{W}=\left\{ X\left(\cdot\right)\right\} $
of random variables was assumed to be separating for $\left(\Omega,\mathscr{F}\right)$. 
\end{proof}
\begin{cor}
Consider the probability space $\left(\Omega,\mathscr{F},\mathbb{P}\right)$
in \eqref{eq:B3} where $\mathscr{F}$ is the product-cylinder $\sigma$-algebra
of subsets in $\Omega$. For $\omega=\left(x_{j}\right)_{j\in\mathbb{N}}\in\Omega$,
set 
\begin{equation}
X_{j}\left(\omega\right)=x_{j}\in\mathbb{R}.
\end{equation}
Then the \uline{algebra} $\mathscr{A}$ generated by 
\begin{equation}
\left\{ \varphi\circ X_{j}\mathrel{;}j\in\mathbb{N},\:\varphi\in\mathscr{C}\right\} \label{eq:B8}
\end{equation}
is uniformly dense in $C\left(\Omega\right)$. 
\end{cor}

\begin{proof}
By Stone-Weierstrass, we only need to verify that the algebra $\mathscr{A}$
from \eqref{eq:B8} \emph{separates points} in $\Omega$. But, if
$\omega\neq\omega'$ in $\Omega$, then there is an index $j$ such
that $X_{j}\left(\omega\right)\neq X_{j}\left(\omega'\right)$. Then
just pick a $\varphi\in\mathscr{C}$ s.t. $\varphi\left(X_{j}\left(\omega\right)\right)\neq\varphi\left(X_{j}\left(\omega'\right)\right)$. 
\end{proof}

\section{\label{sec:ps}Projective space of equivalence classes}
\begin{notation*}
We shall work with projective space $P\left(\mathbb{R}^{m}\right)$,
i.e., equivalence classes in $\mathbb{R}^{m}\backslash\left(0\right)$,
where $w\sim w'\Longleftrightarrow\exists t\in\mathbb{R}\backslash\left(0\right)\text{ s.t. \ensuremath{w'=tw}}$.
Set 
\begin{equation}
\mathscr{H}\left(w\right):=closure\left\{ \varphi\left(w^{T}x\right),\;\varphi\in\mathscr{C}:=C\left(\mathbb{R},\mathbb{C}\right)\right\} .\label{eq:C2}
\end{equation}
The closure in \eqref{eq:C2} is w.r.t the $\left\Vert \cdot\right\Vert _{\infty}$-norm
of $C\left(J^{m}\right)$, or the $L^{2}\left(J^{m}\right)$-norm
on $L^{2}\left(J^{m}\right)$; both cases are of interest. Note, if
$w\sim w'$, then $\mathscr{H}\left(w\right)=\mathscr{H}\left(w'\right)$.
\end{notation*}
We discuss measures $\mu$ of finite total variation on $J^{m}$ such
that 
\begin{align}
\int_{J^{m}}\varphi\left(w^{T}x\right)d\mu\left(x\right) & =0,\quad\forall\varphi\in\mathscr{C}.\label{eq:C3}
\end{align}
And in the $L^{2}\left(J^{m}\right)$ case, we consider $F\in L^{2}\left(J^{m}\right)$
such that 
\begin{equation}
\int_{J^{m}}\varphi\left(w^{T}x\right)F\left(x\right)d^{m}x=0,\quad\forall\varphi\in\mathscr{C}.\label{eq:C4}
\end{equation}

\begin{lem}
Equation \eqref{eq:C4} is included in the condition on $\mu$ from
\eqref{eq:C3}. 
\end{lem}

\begin{proof}
Take $\mu$ of the form 
\begin{equation}
d\mu=F\left(x\right)d^{m}x.\label{eq:C5}
\end{equation}
Viewed as signed measure, we have $d\mu\ll d^{m}x$, where $d^{m}x$
is standard Lebesgue measure on $J^{m}$. 
\end{proof}
\begin{rem}
For all $F\in L^{2}\left(J^{m}\right)\subset L^{1}\left(J^{m}\right)$,
the total variation measure $\left|\mu\right|$ from \eqref{eq:C5}
is $d\left|\mu\right|\left(x\right)=\left|F\left(x\right)\right|d^{m}x$.
Recall that, for all Borel measurable sets $B$ in $J^{m}$, we have
\begin{equation}
\left|\mu\right|\left(B\right):=\sup\left\{ \sum\nolimits _{i}\left|\mu\left(A_{i}\right)\right|\right\} \label{eq:C6}
\end{equation}
where $\left\{ A_{i}\right\} $ runs over all partitions of $B$,
i.e., $B=\cup_{i}A_{i}$, $A_{i}\cap A_{j}=\emptyset$ if $i\neq j$. 

More generally, we shall make use of the following: Let $\Omega$
be a compact space, then $C\left(\Omega\right)$ with norm $\left\Vert \cdot\right\Vert _{\infty}$,
as a Banach space, has for its dual $\mathscr{M}:=\left\{ \text{all signed measures on \ensuremath{\Omega} of finite variation}\right\} $
with $\left\Vert \mu\right\Vert _{*}=\left|\mu\right|\left(\Omega\right)$
(see \eqref{eq:C6}) as the dual norm via $\mu\left(F\right)=\int_{\Omega}F\,d\mu$,
$F\in C\left(\Omega\right)$. 
\end{rem}

\paragraph*{}
\begin{defn}
Set 
\[
\mathscr{A}\left(W\right)=\left\{ \varphi\left(w^{T}x\right)\mathrel{;}w\in W,\;\varphi\in C\left(\mathbb{R},\mathbb{C}\right)\right\} .
\]
\end{defn}

\begin{namedthm}[Observations]
If $W=\left(\pi\mathbb{Z}\right)^{m}=$ Fourier lattice, then $\mathscr{A}\left(W\right)$
is dense in $C\left(J^{m}\right)$, and therefore also dense in $L^{2}\left(J^{m}\right)$.
In this case, we will then only need $\varphi\left(t\right):=e^{it}$,
$t\in\mathbb{R}$; or, in the real case, $\varphi_{c}\left(t\right)=\cos t$,
$\varphi_{s}\left(t\right)=\sin t$. 
\end{namedthm}
\begin{proof}
Follows from Stone-Weierstrass; or Fej\'er-Ces\`aro. If $\varphi\left(t\right)=e^{it}$,
then $\varphi\left(w\cdot x\right)=e^{iw\cdot x}$, $w\in\left(\pi\mathbb{Z}\right)^{m}$,
is the usual Fourier basis. Ces\`aro summation yields $\left\Vert \cdot\right\Vert _{\infty}$
approximation in $C\left(J^{m}\right)$. 
\end{proof}

\section{\label{sec:mf}Multivariable Fourier expansions}

The setting of our approximation problem discussed below is related
to that of Cybenko's \prettyref{thm:B1}, but different. Nonetheless,
the framework of Cybenko's paper serves as motivation for our present
considerations. Below we briefly outline differences, beginning with
the starting point.  

Recall, in Cybenko's setting, there is only one given, and fixed,
continuous function $\varphi$ on the real line $\mathbb{R}$, but
subject to the conditions listed in \eqref{enu:B1-2}. So, given $\varphi$
as in \eqref{enu:B1-2}, we allow variation of all $m$-vectors, and
all translation by real numbers. The conclusion of \prettyref{thm:B1}
yields approximation of all continuous functions on $J^{m}$. 

In the setting below, the starting point is different: We fix a set
$W$ of $m$-vectors, and, as in Cybenko, we ask for best approximation
of functions on $J^{m}$, in $C\left(J^{m}\right)$ or $L^{2}$. But,
in the present setting, we shall allow variation over all bounded
continuous functions $\varphi$ on $\mathbb{R}$. Also the given set
$W$ will now in fact be considered a subset of projective space $P\left(\mathbb{R}^{m}\right)$.
One of our results states that $W$-approximation of classes of functions
on $J^{m}$ will not be possible when $W$ is finite. Hence we shall
also consider countably infinite subsets $W$ of $P\left(\mathbb{R}^{m}\right)$,
and we shall address orthogonal decompositions as well, and an associated
harmonic analysis. Our approach will be constructive.

\paragraph*{Fourier coefficients of functions $F$ in $\mathscr{H}\left(w\right)$,
where $w\in\mathbb{Z}^{m}\backslash\left(0\right)$ is fixed:}

Let $\mathscr{H}\left(w\right)$ be as in \eqref{eq:C2}. Set 
\[
F\left(x\right):=\varphi\left(w^{T}x\right)=\sum_{k\in\mathbb{Z}}a_{k}e^{i\pi kw\cdot x},
\]
where $\varphi\left(s\right)=\sum_{k\in\mathbb{Z}}a_{k}e^{i\pi ks}$,
$s\in\mathbb{R}$, is the 1D Fourier expansion of $\varphi$; so $F$
is supported, in the Fourier domain, by the set $\mathbb{Z}w=\left\{ kw\mathrel{;}k\in\mathbb{Z}\right\} $.
We shall use the usual notation: 
\begin{align*}
kw\cdot x & =kw_{1}x_{1}+kw_{2}x_{2}+\cdots+kw_{m}x_{m},\quad k\in\mathbb{Z},\\
w & =\left(w_{1},\cdots,w_{m}\right)\in\mathbb{Z}^{m}\backslash\left(0\right),\;\text{and}\\
x & =\left(x_{1},\cdots,x_{m}\right)\in\mathbb{R}^{m}.
\end{align*}

Fix $w$. Let $P_{W}$ be the projection from $L^{2}\left(J^{m}\right)$
onto 
\begin{equation}
\mathscr{H}_{W}:=\overline{span}^{L^{2}\left(J^{m}\right)}\left\{ \varphi\left(w^{T}x\right)\mathrel{;}\varphi\in C\left(\mathbb{R},\mathbb{R}\right),\;w\in W\right\} .\label{eq:D1}
\end{equation}
Thus, $P_{W}\left(f\right)$ is the unique $L^{2}\left(J^{m}\right)$-minimizer:
\begin{equation}
\left\Vert f-P_{W}f\right\Vert _{L^{2}}=\inf\left\{ \left\Vert f-g\right\Vert _{L^{2}}\mathrel{;}g\in\mathscr{H}_{W}\right\} .\label{eq:D2}
\end{equation}
However, it is not always easy to find a formula for $P_{W}f$. 
\begin{center}
\includegraphics[width=0.4\textwidth]{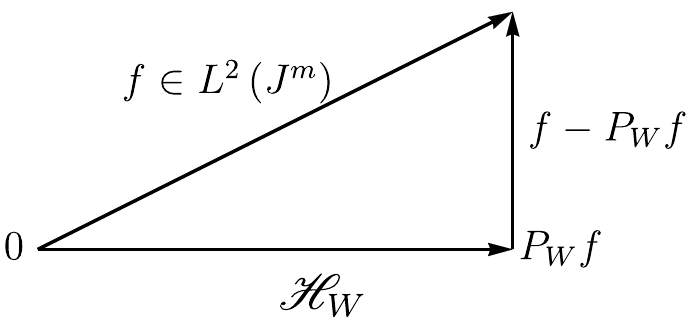}
\par\end{center}
\begin{notation*}
In the case of $L^{2}\left(J^{m}\right)$, we can simply form the
closure of the subspace $\left\{ \varphi\left(w^{T}x\right)\right\} _{w\in W,\,\varphi\in C}$,
i.e., closure in the $L^{2}\left(J^{m}\right)$-norm. We write 
\begin{equation}
W_{2}^{\perp}=\mathscr{H}_{W}^{\perp}:=L^{2}\left(J^{m}\right)\ominus\mathscr{H}_{W}=P_{W}^{\perp}L^{2}\left(J^{m}\right).\label{eq:D2b}
\end{equation}
And so the functions $F\in W_{2}^{\perp}$ are simply the functions
$F$ s.t. $P_{W}\left(F\right)=0$, where $\mathscr{H}_{W}$ is as
in \eqref{eq:D1}, and $P_{W}=P_{\mathscr{H}_{W}}$ is the orthogonal
projection in $L^{2}\left(J^{m}\right)$ onto $\mathscr{H}_{W}$;
see \eqref{eq:D2}. 
\end{notation*}
Fix a subset $W\subset\mathbb{R}^{m}\backslash\left(0\right)$. For
a function $F$ on $J^{m}$ or on $\mathbb{R}^{m}$, consider the
two actions; translation and scaling:
\begin{alignat*}{2}
\left(T_{y}F\right)\left(x\right) & =F\left(x+y\right),\quad &  & y\in\mathbb{R}^{m},\\
\left(S_{a}F\right)\left(x\right) & =F\left(ax\right), &  & a\in\mathbb{R}.
\end{alignat*}
(If $F$ is a function on $J^{m}$, translation is modulo by $2\mathbb{Z}$.) 

The invariance properties are for both the approximation problems
in $C\left(J^{m}\right)$ and in $L^{2}\left(J^{m}\right)$. Below
we recall properties of the operators of translation $T_{y}$ and
of scaling $S_{a}$.
\begin{lem}
Both the subspace $\mathscr{H}_{W}$, and the orthogonal complement
$W_{2}^{\perp}$, are invariant under the two actions $\left\{ T_{y}\right\} $
and $\left\{ S_{a}\right\} $. 
\end{lem}

\begin{defn}
\label{def:D2}Let $\mathscr{M}$ denote the Borel measures $\mu$
on $J^{m}$ of finite total variation.
\end{defn}

In the case of $C\left(J^{m}\right)$, we study $\mu\in\mathscr{M}$
s.t. 
\begin{equation}
\int_{J^{m}}\varphi\left(w^{T}x\right)d\mu\left(x\right)=0,\quad\forall\varphi\in\mathscr{C},\;\forall w\in W,\label{eq:D3}
\end{equation}
where $\mathscr{C}=C_{b}\left(\mathbb{R},\mathbb{R}\right)$, or $C_{b}\left(\mathbb{R},\mathbb{C}\right)$. 

In the case of $L^{2}\left(J^{m}\right)$, we consider $F$ satisfying
\begin{equation}
\int_{J^{m}}F\left(x\right)\varphi\left(w^{T}x\right)d^{m}x=0,\quad\forall\varphi\in\mathscr{C},\;\forall w\in W,\label{eq:D4}
\end{equation}
i.e., $F\perp\left\{ \varphi\left(w^{T}x\right)\right\} _{w\in W}$
in $\ensuremath{L^{2}\left(J^{m}\right)}$. 
\begin{lem}
\label{lem:D3}$\mu\in\mathscr{M}$ satisfies \eqref{eq:D3} iff
\begin{equation}
\widehat{\mu}\left(tw\right)=0,\quad\forall t\in\mathbb{R},\;\forall w\in W.
\end{equation}
The solution \emph{(}in $L^{2}\left(J^{m}\right)$\emph{)} to \eqref{eq:D4}
is 
\begin{equation}
\widehat{F}\left(tw\right)=0,\quad\forall t\in\mathbb{R},\;\forall w\in W.
\end{equation}
Here, $\widehat{F}$ is the following Fourier transform: 
\begin{equation}
\widehat{F}\left(\xi\right):=\int_{J^{m}}e^{i\xi\cdot x}F\left(x\right)d^{m}x,\quad\forall\xi\in\mathbb{R}^{m}.\label{eq:D7}
\end{equation}
\end{lem}

\begin{defn}
For $\mu,\nu\in\mathscr{M}$ (see \prettyref{def:D2}), we denote
by $\mu\ast\nu$ the convolution given by 
\[
\int\varphi\,d\left(\mu\ast\nu\right):=\int_{J^{m}}\int_{J^{m}}\varphi\left(x+y\right)d\mu\left(x\right)d\nu\left(y\right),\quad\varphi\in\mathscr{C}.
\]
Note that $\left(\mu\ast\nu\right)^{\wedge}=\widehat{\mu}\,\widehat{\nu}$,
pointwise product. The algebra $\left\{ \widehat{\mu}\mathrel{;}\mu\in\mathscr{M}\right\} $
is called the Fourier algebra.
\end{defn}

\begin{lem}
\label{lem:Wp}Fix a subset $W\subset\mathbb{R}^{m}\backslash\left(0\right)$.
Then 
\begin{equation}
W_{\mathscr{M}}^{\perp}=\left\{ \mu\in\mathscr{M}\mathrel{;}\widehat{\mu}\left(tw\right)=0,\;\forall t\in\mathbb{R},\:\forall w\in W\right\} \label{eq:D8}
\end{equation}
is an ideal in the convolution algebra. Equivalently, the Fourier
transforms $\left\{ \widehat{\mu}\mathrel{;}\mu\in W_{\mathscr{M}}^{\perp}\right\} $
is an \uline{ideal} in the Fourier algebra. 
\end{lem}

\begin{proof}
Immediate from the definitions. 
\end{proof}
\begin{rem}[Analytic continuation of $\widehat{\mu}$ and $\widehat{F}$]
Note that both $\widehat{\mu}\left(\cdot\right)$ and $\widehat{F}\left(\cdot\right)$
are entire analytic, and so extend to $\mathbb{C}^{m}$; $\xi\in\mathbb{R}^{m}\longrightarrow\mathbb{C}^{m}$.

The notation of the function $\widehat{F}\left(\xi\right)$ in \eqref{eq:D7},
$\xi\in\mathbb{R}^{m}$, is reasonably well understood. The extension
from $\mathbb{R}^{m}$ to $\mathbb{C}^{m}$ is as follows: Set 
\[
\widehat{F}\left(\zeta\right)=\underset{\widehat{G}\left(\zeta\right)}{\underbrace{\int_{J^{m}}e^{i\left(x_{1}\zeta_{1}+\cdots+x_{m}\zeta_{m}\right)}F\left(x\right)d^{m}x}}
\]
where $\widehat{G}\left(\zeta\right)$ is an entire analytic function
of exponential type, i.e., 
\[
|\widehat{G}\left(\zeta\right)|\leq const\:e^{2\sum_{1}^{m}\left|\Im\zeta_{j}\right|}.
\]
This is known as a Paley-Wiener class, and is reasonably well understood;
see e.g., books by H\"{o}rmander \cite{MR1996773,MR2108588}. 
\end{rem}

\section{A Radon Transform}

Recall that the Radon transform (see e.g., \cite{MR3405370}) is an
integral transform taking a function $f$ defined on the plane to
a function $Rf$ defined on the (two-dimensional) space of lines in
the plane, whose value at a particular line is equal to the line integral
of the function over that line. Below we need a higher dimensional
variant (see \prettyref{lem:E1}) of this idea, and we shall refer
to it also as a Radon transform.
\begin{lem}[Radon transform]
\label{lem:E1}With the measure in $L^{2}\left(J^{m}\right)$, we
obtain explicit formulas: Fix $W\subset\mathbb{R}^{m}\backslash\left(0\right)$,
and for $w\in W$, set 
\begin{equation}
\Pi_{w}=\left\{ y\in\mathbb{R}^{m}\mathrel{;}w^{T}y=0\right\} ,
\end{equation}
i.e., the hyperplane\emph{ (}\prettyref{fig:E1}\emph{)}; then 
\begin{gather}
F\in L^{2}\left(J^{m}\right)\ominus\mathscr{H}_{W}\\
\Updownarrow\nonumber \\
\int_{y\in\Pi_{w}}F\left(sw+y\right)d\sigma_{w}\left(y\right)=0,\;\forall w\in W,\:\forall s\in\mathbb{R},\label{eq:E17}
\end{gather}
where $d\sigma_{w}=d\sigma_{w}^{m-1}$ is the standard Lebesgue measure
on $\Pi_{w}\simeq\mathbb{R}^{m-1}$. Note that \eqref{eq:E17} is
a Radon-transform. (See, e.g., \cite{MR3737460,MR3804007}.) 
\end{lem}

\begin{center}
\begin{figure}
\begin{centering}
\includegraphics[width=0.4\textwidth]{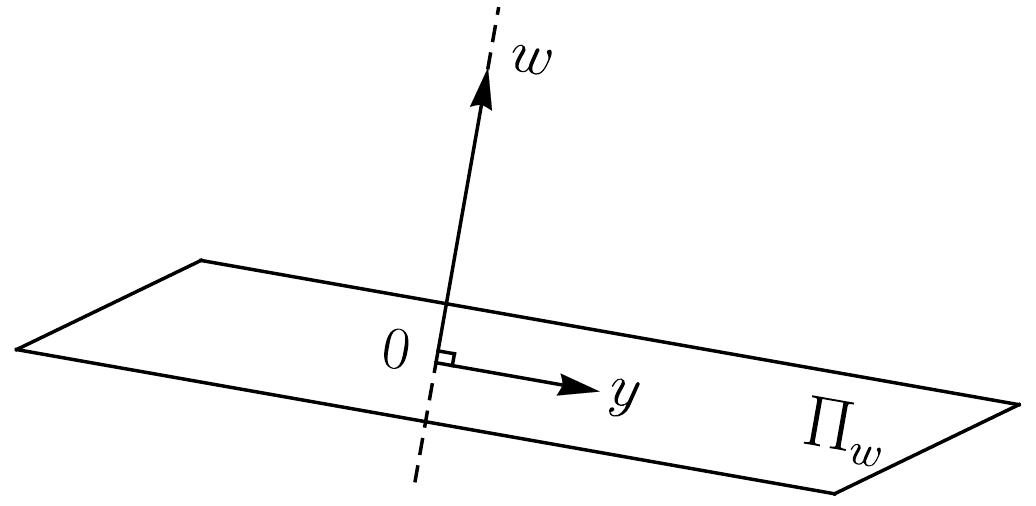}
\par\end{centering}
\caption{\label{fig:E1}The hyperplane $\Pi_{w}$}
\end{figure}
\par\end{center}
\begin{proof}
Fix $W\subset\mathbb{R}^{m}\backslash\left(0\right)$. Assume without
loss of generality that $\left\Vert w\right\Vert _{2}=1$. For $x\in\mathbb{R}^{m}$,
set $y=x-\left(w^{T}x\right)w$, then $y\in\Pi_{w}$ i.e., $w^{T}y=0$
by a direct calculation. 

Introduce a coordinate system $\mathbb{R}\times\Pi_{w}\longleftrightarrow\mathbb{R}^{m}$,
\[
\left(s,y\right)\longmapsto x=sw+y\in\mathbb{R}^{m}
\]
with $s\in\mathbb{R}$, $y\in\Pi_{w}$ ($w$ is fixed and normalized);
then 
\[
\int\varphi\left(w^{T}x\right)F\left(x\right)d^{m}x=\int\varphi\left(s\right)\left(\int_{y\in\Pi_{w}}F\left(sw+y\right)d\sigma_{w}\left(y\right)\right)ds.
\]
It follows that
\[
F\in W_{2}^{\perp}\Longleftrightarrow\int_{y\in\Pi_{w}}F\left(sw+y\right)d\sigma_{w}\left(y\right)=0,\quad\forall w\in W,\:\forall s\in\mathbb{R}.
\]
\end{proof}
\begin{cor}
For $w\in\mathbb{R}^{m}\backslash\left(0\right)$, let $\Pi_{w}$
and $d\sigma_{w}$ be as above. Define the following operator (Radon
transform) $R_{w}:L^{2}\left(J^{m}\right)\longrightarrow L_{loc}^{2}\left(\mathbb{R}\right)$,
\begin{equation}
\left(R_{w}\left(F\right)\right)\left(t\right):=\int_{y\in\Pi_{w}}F\left(tw+y\right)d\sigma_{w}\left(y\right),\quad t\in\mathbb{R},\label{eq:E4}
\end{equation}
then 
\[
F\in L^{2}\left(J^{m}\right)\ominus\left\{ \varphi\left(w^{T}x\right)\right\} _{\varphi\in\mathscr{C},\,w\in W}\Longleftrightarrow R_{w}\left(F\right)\equiv0.
\]
\end{cor}

\begin{proof}
We have 
\[
\int_{\mathbb{R}}\varphi\left(t\right)\left(R_{w}F\right)\left(t\right)dt=\left(Jac\right)\int_{J^{m}}\varphi\left(w^{T}x\right)F\left(x\right)d^{m}x,
\]
where ``$Jac$'' denotes the corresponding Jacobian.
\end{proof}
Here is another corollary of the duality approach: 
\begin{cor}
If $W\subset\mathbb{R}^{m}\backslash\left(0\right)$ is given and
finite, then 
\[
F\in L^{2}\left(J^{m}\right)\ominus\left\{ \varphi\left(w^{T}x\right)\right\} _{\varphi\in\mathscr{C},\,w\in W}\Longleftrightarrow R_{w}F=0,\;\forall w\in W
\]
is infinite dimensional. 

It follows in particular that the space of solutions $\mu$ to 
\[
\int_{J^{m}}\varphi\left(w^{T}x\right)d\mu\left(x\right)=0,\quad\forall\varphi\in\mathscr{C},\:\forall w\in W,
\]
is infinite-dimensional \emph{(}i.e., $\mu\in W_{\mathscr{M}}^{\perp}$,
see \eqref{eq:D8}\emph{)}. 
\end{cor}

Below is a property that holds for functions $\varphi\left(w^{T}x\right)$
and not for other functions $F$ in $C\left(J^{m}\right)$ or in $L^{2}\left(J^{m}\right)$: 
\begin{lem}
Fix $w$ and $\varphi$, and set $F\left(x\right)=\varphi\left(w^{T}x\right)$,
then $F$ is constant on every hyperplane 
\[
\Pi_{w,t}:=\left\{ x\in\mathbb{R}^{m}\mathrel{;}w^{T}x=t\right\} .
\]
\end{lem}

\begin{center}
\includegraphics[width=0.4\textwidth]{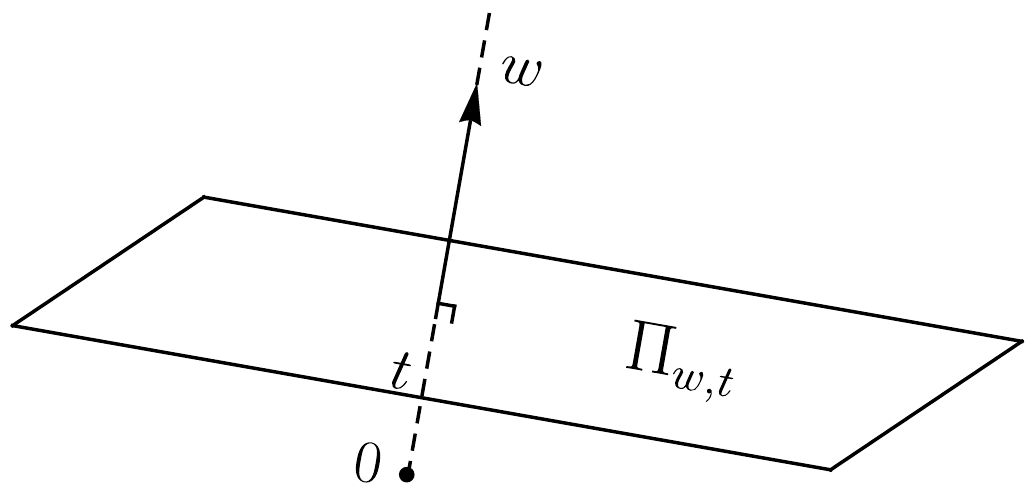}
\par\end{center}
\begin{proof}
If $F\left(x\right)=\varphi\left(w^{T}x\right)$, $x\in\Pi_{w,t}$,
then $F\left(x\right)=\varphi\left(t\right)$. But we will need to
also compute $\varphi\left(w^{T}x\right)$ when $w\neq w_{0}$ in
$\mathbb{R}^{m}$ and $x\in\Pi_{w_{0},t_{0}}$. The function is not
constant on $\Pi_{w_{0},t}$ but it depends on only one angle. 

In details, let $w\neq w_{0}$ be as above, and suppose $x\in\Pi_{w_{0},t_{0}}$.
Write 
\[
w=\alpha w_{0}+w^{T}
\]
where $w^{T}\cdot w_{0}=0$. Then $\varphi\left(w\cdot x\right)=\varphi\left(t_{0}\alpha+w^{\perp}\cdot x\right)$. 
\end{proof}
\begin{cor}
Let $W\subset\mathbb{R}^{m}\backslash\left(0\right)$ be a finite
subset, then 
\[
\left\{ \varphi\left(w^{T}x\right)\right\} _{\varphi\in\mathscr{C},\,w\in W}
\]
is \uline{not} dense in $L^{2}\left(J^{m}\right)$.
\end{cor}

\begin{proof}
For every $w\in W$, let $\left\{ R_{w}\right\} _{w\in W}$ be the
system of Radon transforms in $L^{2}\left(J^{m}\right)$, see \eqref{eq:E4},
i.e., 
\[
R_{w}\left(F\right)\left(t\right)=\int_{y\in\Pi_{w}}F\left(tw+y\right)d\sigma_{w}^{\left(m-1\right)}\left(y\right),\quad t\in\mathbb{R}.
\]
We consider function $G\in ker\left(R_{w}\right)$, so $G\perp_{L^{2}\left(J^{m}\right)}\mathscr{H}_{W}$.
Suppose $W=\left\{ w_{i}\right\} _{i=1}^{p}$, then the functions
\begin{equation}
K:=G_{w_{1}}\ast G_{w_{2}}\ast\cdots\ast G_{w_{p}}\perp_{L^{2}\left(J^{m}\right)}\mathscr{H}_{W},\label{eq:E5}
\end{equation}
as $G_{w_{i}}\in ker\left(R_{w_{i}}\right)$ are chosen. Then $K$
is in $L^{2}\left(J^{m}\right)\ominus\mathscr{H}_{W}$. The operation
$\ast$ in \eqref{eq:E5} denotes convolution. 
\end{proof}

\section{Reproducing kernel and Shannon interpolation}

Starting with $m=2$, we shall display a complete list of points $w\in\mathbb{Z}^{2}\backslash\left(0\right)$
such that the corresponding subspaces $\mathscr{H}'\left(w\right):=\mathscr{H}\left(w\right)\ominus\mathbb{C}\mathbbm{1}$
are mutually orthogonal. There is also the analogous question for
$\mathbb{Z}^{m}$, $m>2$. The trick is to make a list of points $w$
in $\mathbb{Z}^{m}\backslash\left(0\right)$ such that the integer
multiples $nw$, $n\in\mathbb{Z}$ (i.e., integer lines), cover $\mathbb{Z}^{m}$
with no overlap other than in 0. The partitions of $\mathbb{Z}^{m}$
corresponds to equivalence classes in $\mathbb{Z}^{m}$, hence non-overlap.
We then make a system of orthogonal subspaces $\mathscr{H}'\left(w\right)$
which is also total in $L^{2}\left(J^{m}\right)$. This is made precise
in \prettyref{lem:F11}; also see Examples \ref{exa:F12}--\ref{exa:F13}.\\

The functions in \eqref{eq:D7} assume the multivariable Shannon interpolation,
and $\{\widehat{F}\}_{F\in L^{2}\left(J^{m}\right)}$ is a reproducing
kernel Hilbert space (RKHS) with the following Hilbert norm: 
\begin{equation}
\Vert\widehat{F}\Vert_{RKHS}^{2}:=\left\Vert F\right\Vert _{L^{2}\left(J^{m}\right)}^{2}=\int_{J^{m}}\left|F\left(x\right)\right|^{2}d^{m}x.
\end{equation}
See, e.g., \cite{MR0233416,MR766625,MR1036623,MR1200633,MR1225669}.

The case $m>1$, $L^{2}\left(J^{m}\right)$, leads to a multivariable
Shannon interpolation for the Fourier transform $\widehat{F}\left(\xi\right):=\int_{J^{m}}F\left(x\right)e^{ix\cdot\xi}d^{m}x$,
$\xi\in\mathbb{R}^{m}$:
\begin{align}
\widehat{F}\left(\xi\right) & =\sum_{\lambda\in\left(\mathbb{Z}/\pi\right)^{m}}\widehat{F}\left(\lambda\right)K_{m}\left(\xi-\lambda\right),\quad\forall\xi\in\mathbb{R}^{m},\label{eq:F2}
\end{align}
where $\left(\mathbb{Z}/\pi\right)^{m}$ is the dual lattice, and
$\widehat{F}\left(\lambda\right)$ are the Fourier coefficients.
\begin{defn}
Fix $m$, and $W\subset\mathbb{R}^{m}\backslash\left(0\right)$ a
finite subset. Set 
\begin{align}
W_{\mathscr{M}}^{\perp} & :=\big\{\text{signed finite total variation measure \ensuremath{\mu} on \ensuremath{J^{m}} s.t.}\nonumber \\
 & \qquad\ensuremath{\int\varphi\left(w^{T}x\right)d\mu\left(x\right)=0,\;\forall\varphi\in C_{b}\left(\mathbb{R},\mathbb{R}\right),}\text{i.e.,}\nonumber \\
 & \qquad\widehat{\mu}\left(tw\right)=0,\;\forall t\in\mathbb{R},\:\forall w\in W\big\}\label{eq:F4}
\end{align}
and 
\begin{align}
W_{2}^{\perp} & :=\big\{ F\in L^{2}\left(J^{m}\right)\mathrel{;}\int_{J^{m}}\varphi\left(w^{T}x\right)F\left(x\right)d^{m}x=0,\;\forall\varphi,\;\text{i.e., }\nonumber \\
 & \qquad\widehat{F}\left(tw\right)=0,\;\forall t\in\mathbb{R},\;\forall w\in W\big\}.\label{eq:F5}
\end{align}
The orthogonal complement ``$\perp$'' in \eqref{eq:F5} refers
to $L^{2}\left(J^{m}\right)$. Note that $W_{2}^{\perp}\subset W_{\mathscr{M}}^{\perp}$
since, if $F\in L^{2}\left(J^{m}\right)$, $d\mu\left(x\right)=F\left(x\right)d^{m}x$
is in $\mathscr{M}$. (Also see \eqref{eq:D2b}, and \prettyref{lem:Wp}.)
\end{defn}

If we study approximations in $L^{2}\left(J^{m}\right)$, then the
question is: Fix $F\in L^{2}\left(J^{m}\right)$ s.t. $\left\langle F,\varphi\left(w^{T}x\right)\right\rangle _{L^{2}\left(J^{m}\right)}=0$,
$\forall w\in W$, $\forall\varphi$. We shall also consider signed
measures $\mu$ s.t. $d\mu\left(x\right)=F\left(x\right)d^{m}x$,
then each condition for $\mu\in W_{\mathscr{M}}^{\perp}$ translates
into $\widehat{F}\left(tw\right)=0$, $\forall w\in W$, $t\in\mathbb{R}$.
See \prettyref{lem:D3}.

More specifically, it follows from \eqref{eq:F2} that
\begin{equation}
\widehat{F}\left(tw\right)=\sum_{\lambda\in\left(\mathbb{Z}/\pi\right)^{m}}\widehat{F}\left(\lambda\right)K_{m}\left(tw-\lambda\right)\label{eq:F3}
\end{equation}
Equation \eqref{eq:F3} is entire analytic; and our condition takes
the form: 
\begin{gather*}
\widehat{F}\left(tw\right)=0,\quad\forall t\in\mathbb{R}\\
\Updownarrow\\
\left(\frac{d}{dt}\right)^{k}\widehat{F}\left(tw\right)\Big|_{t=0}=0,\quad\forall k\in\mathbb{N}_{0}\\
\Updownarrow\\
\sum_{\lambda\in\left(\mathbb{Z}/\pi\right)^{m}}\widehat{F}\left(\lambda\right)K_{m}^{\left(k\right)}\left(w-\lambda\right)=0,\\
\forall w\in W,\:\forall k\in\mathbb{N}_{0}.
\end{gather*}

\begin{lem}
Let $w,w'\in\mathbb{Z}^{m}\backslash\left(0\right)$, and assume $w\neq w'$;
then the following orthogonality relation holds: 
\begin{equation}
\underset{\left\langle e_{w},e_{w'}\right\rangle _{L^{2}\left(J^{m}\right)}}{\underbrace{\frac{1}{2^{m}}\int_{J^{m}}e^{i\pi w^{T}x}\overline{e^{i\pi w'^{T}x}}\,d^{m}x}}=\underset{\delta\left(w-w'\right)}{\underbrace{\vphantom{\frac{1}{2^{m}}\int_{J^{m}}}\delta\left(w_{1}-w'_{1}\right)\delta\left(w_{2}-w'_{2}\right)\cdots\delta\left(w_{m}-w'_{m}\right)}.}\label{eq:F6}
\end{equation}
 
\end{lem}

\begin{proof}
$L^{2}$-inner products: 
\begin{align*}
\text{LHS}_{\left(\ref{eq:F6}\right)} & =\frac{1}{2^{m}}\left(\int_{-1}^{1}e^{i\pi\left(w_{1}-w_{1}'\right)x_{1}}dx_{1}\right)\left(\int_{-1}^{1}e^{i\pi\left(w_{2}-w_{2}'\right)x_{2}}dx_{2}\right)\cdots\\
 & \qquad\cdots\left(\int_{-1}^{1}e^{i\pi\left(w_{m}-w_{m}'\right)x_{m}}dx_{m}\right)\\
 & =\delta\left(w_{1}-w'_{1}\right)\delta\left(w_{2}-w'_{2}\right)\cdots\delta\left(w_{m}-w'_{m}\right)\\
 & =\delta\left(w-w'\right)\quad\left(\text{in the abbreviated notation.}\right)
\end{align*}
\end{proof}
\begin{lem}
If $\varphi,\psi\in\mathscr{C}$, we also get orthogonality when $w\neq w'$
and inequivalent, then 
\begin{equation}
\int_{J^{m}}\varphi\left(w^{T}x\right)\overline{\psi\left(w'^{T}x\right)}d^{m}x=0\label{eq:F7}
\end{equation}
unless $\varphi$ and $\psi$ contain constant components.
\end{lem}

\begin{proof}
Note the assumption is that $w$ and $w'$ are inequivalent, so $kw\neq lw'$,
$\forall k,l\in\mathbb{Z}\backslash\left(0\right)$.

Use standard Fourier expansions for the two functions $\varphi$ and
$\psi$: 
\begin{equation}
\varphi\left(s\right)=\sum_{k\in\mathbb{Z}}a_{k}e^{i\pi ks},\quad\psi\left(s\right)=\sum_{k\in\mathbb{Z}}b_{k}e^{i\pi ks},\quad s\in\mathbb{R}\label{eq:F8}
\end{equation}
with Fourier coefficients $\left(a_{k}\right)_{k\in\mathbb{Z}}$ and
$\left(b_{k}\right)_{k\in\mathbb{Z}}$. Now substitute \eqref{eq:F8}
into \eqref{eq:F7}, 
\begin{equation}
\varphi\left(w^{T}x\right)=\sum_{k\in\mathbb{Z}}a_{k}e^{i\pi kw\cdot x},\;\text{and}\quad\psi\left(w^{T}x\right)=\sum_{l\in\mathbb{Z}}b_{l}e^{i\pi lw'\cdot x},\quad x\in\mathbb{R}^{m},\label{eq:F9}
\end{equation}
then 
\begin{align*}
\left\langle \varphi\left(w^{T}\cdot\right),\psi\left(w'^{T}\cdot\right)\right\rangle _{L^{2}\left(J^{m}\right)} & =\underset{k,\,l\in\mathbb{Z}}{\sum\sum}a_{k}\overline{b}_{l}\underset{=\delta\left(kw-lw'\right)}{\underbrace{\int_{J^{m}}e^{i\pi\left(kw-lw'\right)\cdot x}d^{m}x}},
\end{align*}
which vanishes unless $k=l=0$, and the latter correspond to the constant
functions; see \eqref{eq:F9}, i.e., $a_{k}=const\:\delta\left(k-0\right)$,
$b_{l}=const\:\delta\left(l-0\right)$, $k,l\in\mathbb{Z}$. 
\end{proof}
In addition to the specific functions $F\in W_{2}^{\perp}$ we list
above (covering some configurations), there are many more. Now we
give a characterization which is based on orthogonality relations. 
\begin{example}[orthogonality, $m=2$]
\label{exa:F3} If $m=2$, then $F\left(x,y\right)=xy\in W_{2}^{\perp}$
where $W:=\left\{ \left(\begin{smallmatrix}1\\
0
\end{smallmatrix}\right),\left(\begin{smallmatrix}0\\
1
\end{smallmatrix}\right)\right\} $. To see this, one checks directly that 
\begin{eqnarray*}
 &  & \int_{-1}^{1}\int_{-1}^{1}\left(f\left(x\right)+g\left(y\right)\right)xy\,\underset{d^{2}x}{\underbrace{dxdy}}\\
 & = & \left\langle f\left(x\right)+g\left(y\right),F\right\rangle _{L^{2}\left(J^{2}\right)}\\
 & = & \left(\int_{-1}^{1}xf\left(x\right)dx\right)\underset{=0}{\underbrace{\left(\int_{-1}^{1}ydy\right)}}+\underset{=0}{\underbrace{\left(\int_{-1}^{1}xdx\right)}}\left(\int_{-1}^{1}yf\left(y\right)dy\right)=0.
\end{eqnarray*}
Hence $F=xy$ is orthogonal to $\mathscr{H}_{W}$. So $P_{W}\left(xy\right)=0$,
and 
\[
\inf_{g\in\mathscr{H}_{W}}\left\Vert xy-g\right\Vert _{L^{2}}^{2}=\left\Vert xy\right\Vert _{L^{2}}^{2}=\left(\frac{2}{3}\right)^{2}.
\]
By the same argument, if $F\left(x,y\right)=\varphi\left(x\right)\psi\left(y\right)$,
assumed nonzero, where $\varphi$ and $\psi$ are odd functions, then
$P_{W}F=0$, and so 
\[
\inf_{g\in\mathscr{H}_{W}}\left\Vert F-g\right\Vert _{L^{2}}^{2}=\left\Vert F\right\Vert _{L^{2}}^{2}=\left(\int_{-1}^{1}\varphi^{2}dx\right)\left(\int_{-1}^{1}\psi^{2}dy\right)>0.
\]
In particular, $W_{2}^{\perp}$ is infinite dimensional.
\end{example}

\begin{example}[$m=2$]
 Let $W:=\left\{ \left(\begin{smallmatrix}1\\
0
\end{smallmatrix}\right),\left(\begin{smallmatrix}0\\
1
\end{smallmatrix}\right)\right\} $, and 
\[
\mathscr{A}_{W}=\left\{ f\left(x\right)+g\left(y\right)\mathrel{;}f,g\in\mathscr{C}\right\} .
\]
Then 
\begin{align}
L^{2}\left(J^{2}\right)\ominus\mathscr{A}_{W} & =\Big\{ F\in L^{2}\left(J^{2}\right)\;\text{s.t.}\int_{-1}^{1}F\left(x,\cdot\right)dx=0,\\
 & \qquad\text{and }\int_{-1}^{1}F\left(\cdot,y\right)dy=0\Big\}.\nonumber 
\end{align}
The functions $F$ may be written in terms of the Fourier expansions:
\begin{equation}
F=\underset{k,\,n\in\mathbb{Z}}{\sum\sum}c_{k,n}e^{i\pi\left(kx+ny\right)},\quad\left\Vert F\right\Vert _{L^{2}}^{2}=\underset{k,\,n\in\mathbb{Z}}{\sum\sum}\left|c_{k,n}\right|^{2}.\label{eq:F11}
\end{equation}
Then 
\begin{gather}
f\in L^{2}\left(J^{2}\right)\ominus\mathscr{A}_{W}\nonumber \\
\Updownarrow\label{eq:F12}\\
c_{0,n}=0,\;c_{k,0}=0,\quad\forall n,k\in\mathbb{Z}.\nonumber 
\end{gather}
\end{example}

\begin{proof}
Compute the marginal Fourier coefficients in \eqref{eq:F11},
\[
\int_{-1}^{1}F\left(x,y\right)dx=\sum_{n\in\mathbb{Z}}c_{0,n}e^{i\pi ny}=0,\quad\forall y
\]
and 
\[
\int_{-1}^{1}F\left(x,y\right)dy=\sum_{k\in\mathbb{Z}}c_{k,0}e^{i\pi kx}=0,\quad\forall x,
\]
and \eqref{eq:F12} follows.
\end{proof}
Hence we can rewrite all questions in terms of Fourier coefficients
$c:=\left\{ c_{k,n}\right\} _{k,n\in\mathbb{Z}}$, 
\[
F\longleftrightarrow c_{k,n}\underset{\text{Defn}}{\Longleftrightarrow}F\left(x,y\right)=\underset{k,n}{\sum\sum}c_{k,n}e^{i\pi\left(kx+ny\right)}.
\]
And we have
\begin{align*}
P_{W}\left(c\right) & =span\left\{ c_{k,0},c_{0,n}\right\} _{k,n\in\mathbb{Z}}\\
 & =\begin{bmatrix} &  &  & \vdots\\
 &  &  & c_{0,2}\\
 &  &  & c_{0,1}\\
\cdots & c_{-2,0} & c_{-1,0} & c_{0,0} & c_{1,0} & c_{2,0} & \cdots\\
 &  &  & c_{0,-1}\\
 &  &  & c_{0,-2}\\
 &  &  & \vdots
\end{bmatrix}
\end{align*}
where $P_{W}$ is the projection onto $\mathscr{H}_{W}$ inside $L^{2}\left(J^{2}\right)$. 
\begin{example}
Use of orthogonality of $\left\{ e^{i\pi\left(kx+ny\right)}\right\} _{\left(k,n\right)\in\mathbb{Z}^{2}}$.
For example, if $F=Ae^{i\pi x}+Be^{i\pi\left(x+y\right)}$ then 
\[
dist\left(F,\mathscr{H}_{W}\right)=\inf_{g\in\mathscr{H}_{W}}\left\Vert F-g\right\Vert _{L^{2}\left(J^{2}\right)}=\left\Vert P_{W}^{\perp}F\right\Vert _{L^{2}\left(J^{2}\right)}=\left|B\right|;
\]
since $P_{W}\left(F\right)=Ae^{i\pi x}$, and $P_{W}^{\perp}\left(F\right)=Be^{i\pi\left(x+y\right)}$.
\end{example}

The following computation works more generally for $\varphi\left(w^{T}x\right)$,
$\varphi\in C\left(\mathbb{R},\mathbb{C}\right)$, $x=\left(x_{1},\cdots,x_{m}\right)$,
$w\in\mathbb{R}^{m}\backslash\left(0\right)$ fixed. But it is helpful
to specialize to $m=2$, $\left(x,y\right)\in\mathbb{R}^{2}$, and
$w=\left(\begin{smallmatrix}1\\
0
\end{smallmatrix}\right)=e_{1}$; we must then compute the Fourier coefficients of sum $e_{1}^{T}\left(x,y\right)=x$. 

In 2D, $F\left(x,y\right)=\varphi\left(x\right)$, with Fourier coefficients
indexed by $\mathbb{Z}^{2}$: 
\begin{align*}
c\left(n,k\right) & =\frac{1}{4}\iint_{J^{2}}\varphi\left(x\right)e^{-i\pi\left(nx+ky\right)}dxdy\\
 & =\underset{\widehat{\varphi}\left(n\right)}{\underbrace{\frac{1}{2}\int_{-1}^{1}\varphi\left(x\right)e^{-i\pi nx}dx}}\:\underset{\delta\left(0-k\right)}{\underbrace{\frac{1}{2}\int_{-1}^{1}e^{-i\pi ky}dy}}\\
 & =\widehat{\varphi}\left(n\right)\delta\left(0-k\right),\quad\forall\left(n,k\right)\in\mathbb{Z}^{2},
\end{align*}
where $\delta\left(0-k\right)=\begin{cases}
1 & \text{if \ensuremath{k=0}}\\
0 & \text{if \ensuremath{k\in\mathbb{Z}\backslash\left(0\right)}}
\end{cases}$. 

Conclusion: $c\left(n,k\right)=\widehat{\varphi}\left(n\right)\delta\left(0-k\right)$,
$\forall\left(n,k\right)\in\mathbb{Z}^{2}$, and 
\[
F\left(x,y\right)=\varphi\left(x\right)=\sum_{n\in\mathbb{Z}}\widehat{\varphi}\left(n\right)e^{i\pi nx}
\]
which is the usual 1-dimensional Fourier expansion. This is a special
case of a more general formula: 

Consider $\mathbb{R}^{m}$, $J^{m}$, $W\subset\mathbb{R}^{m}\backslash\left(0\right)$.
Assume $w\in\mathbb{Z}^{m}$, $w\in W$ fixed. Let $F\left(x\right)=\varphi\left(w^{T}x\right)$,
$x\in J^{m}$, $\varphi:\mathbb{R}\longrightarrow\mathbb{C}$ (or
$\mathbb{R}\longrightarrow\mathbb{R}$). Note $\varphi$ is a function
on one coordinate. Without loss of generality, assume $\left\Vert w\right\Vert =1$
and let $P_{w}$ be the projection 
\begin{align*}
P_{w}\left(x\right) & =\left(w^{T}x\right)w,\;\text{and}\\
P_{w}\left(\mathbb{Z}^{d}\right) & =\left\{ \left(w^{T}n\right)w\mathrel{;}n\in\mathbb{Z}^{d}\right\} 
\end{align*}
then 
\[
F\left(x\right)=\sum_{\lambda\in P_{w}\left(\mathbb{Z}^{d}\right)}\widehat{\varphi}\left(\lambda\right)e^{i\pi\lambda P_{w}\left(x\right)}.
\]
So if we select $w\in\mathbb{Z}^{m}\backslash\left(0\right)$, then
functions in $\mathscr{H}_{w}$ may give the Fourier expansion $F\left(x\right)=\varphi\left(w^{T}x\right)$,
\[
F\left(x\right)=\sum_{k\in\mathbb{Z}}\widehat{\varphi}\left(k\right)e^{ikw^{T}x}
\]
up to normalization. In the calculation of the 1D Fourier coefficients,
\begin{align*}
\widehat{\varphi}\left(k\right) & =\int_{-1}^{1}\varphi\left(x\right)e^{-i\pi ks}ds,\;\text{and}\\
\varphi\left(s\right) & =\sum_{k\in\mathbb{Z}}\widehat{\varphi}\left(k\right)e^{i\pi ks},\quad s\in\mathbb{R},\:k\in\mathbb{Z}.
\end{align*}
In the general case, $w_{1},\cdots,w_{p}\in\mathbb{R}^{m}\backslash\left(0\right)$,
and may assume independent, and also a choice of $w_{k}\in\mathbb{Z}^{m}$;
for $F\in\mathscr{H}_{w}$,
\[
F\left(x\right)=\varphi_{1}\left(w_{1}^{T}x\right)+\varphi_{2}\left(w_{2}^{T}x\right)+\cdots+\varphi_{p}\left(w_{p}^{T}x\right)
\]
$\varphi_{1},\varphi_{2},\cdots,\varphi_{p}\in\mathscr{C}:=C_{b}\left(\mathbb{R},\mathbb{C}\right)$,
or $C_{b}\left(\mathbb{R},\mathbb{R}\right)$. After a renormalization,
\begin{align*}
F\left(x\right) & =\sum_{k\in\mathbb{Z}}\widehat{\varphi}_{1}\left(k\right)e^{ikw_{1}^{T}x}+\cdots+\sum_{k\in\mathbb{Z}}\widehat{\varphi}_{p}\left(k\right)e^{ikw_{p}^{T}x}\\
 & =\sum_{\left(k_{1},\cdots,k_{p}\right)\in\mathbb{Z}^{m}}\sum_{j=1}^{p}\widehat{\varphi}_{j}\left(k_{j}\right)e^{i\pi k_{j}w_{j}^{T}x}.
\end{align*}

\begin{lem}
If $w,w'\in P\left(\mathbb{R}^{m}\right)$ are distinct, equivalent
class, then assume $w$ and $w'$ both rational. The two subspaces
$\mathscr{H}_{w}$ and $\mathscr{H}_{w'}$ in $L^{2}\left(J^{m}\right)$
are orthogonal, i.e., 
\begin{equation}
\int_{J^{m}}F\left(x\right)\overline{F'\left(x\right)}d^{m}x=0,\quad\forall F\in\mathscr{H}_{w},\;\forall F'\in\mathscr{H}_{w'}\label{eq:F15}
\end{equation}
 unless the functions are constant. 
\end{lem}

\begin{proof}
Select $w,w'\in\mathbb{Z}^{m}\backslash\left(0\right)$ and compute
the Fourier expansions of the two functions, $F\left(x\right)$ with
coefficients in $\mathbb{Z}w$, and $F'\left(x\right)$ with coefficients
in $\mathbb{Z}w'$. But since $w$ and $w'$ are inequivalent, 
\begin{equation}
\mathbb{Z}w\bigcap\mathbb{Z}w'=0\;\text{in \ensuremath{\mathbb{Z}^{m}}}
\end{equation}
and so the inner product in \eqref{eq:F15} $\equiv0$ unless the
two functions $F\in\mathscr{H}_{w}$, and $F'\in\mathscr{H}_{w'}$
are constant. 
\end{proof}
As $w$ varies over $\mathbb{Z}^{m}\backslash\left(0\right)$, we
get a system of orthogonal subspaces ``nearly orthogonal'' and if
$w$ and $w'$ are inequivalent, 
\begin{align*}
\mathscr{H}\left(w\right)\cap\mathscr{H}\left(w'\right) & =\text{constant multiples of the function \ensuremath{\mathbbm{1}}}
\end{align*}
and $\mathscr{H}\left(w\right)\perp\mathscr{H}\left(w'\right)$ except
for the constants. Recall, 
\[
\mathscr{H}\left(w\right):=class^{L^{2}\left(J^{m}\right)}\left\{ \varphi\left(w^{T}x\right)\mathrel{;}\varphi\in C\left(\mathbb{R},\mathbb{C}\right)\right\} .
\]

An illustration of the subspaces in the case of $m=2$, and $w\in\mathbb{Z}^{2}\backslash\left(0\right)$.
The property orthogonality for the subspace $\mathscr{H}'\left(w\right):=\mathscr{H}\left(w\right)\ominus\mathbb{C}\mathbbm{1}$,
where $\mathbbm{1}$ is the constant function $\mathbbm{1}$ on $J^{m}$.
Hence 
\[
F\in\mathscr{H}'\left(w\right)\Longleftrightarrow\int_{J^{m}}F\left(x\right)d^{m}x=0.
\]
The argument above shows that 
\begin{equation}
\mathscr{H}'\left(w\right)\perp\mathscr{H}'\left(w'\right)
\end{equation}
where $w$ and $w'$ are inequivalent, so that 
\begin{equation}
\int_{J^{m}}F\left(x\right)\overline{F'\left(x\right)}d^{m}x=0,\quad\forall F\in\mathscr{H}'\left(w\right),\:\forall F'\in\mathscr{H}'\left(w'\right).
\end{equation}

Below is a set of independent equivalent classes (and the subspaces
are orthogonal), $k\in\mathbb{Z}\backslash\left(0\right)$ fixed. 
\begin{center}
\begin{tabular}{|c|c|c|}
\hline 
$\left(0,k\right)$ & $k\in\mathbb{Z}$ & $class\left(0,1\right)$: $\mathscr{H}\left(0,1\right)$\tabularnewline
\hline 
$\left(n,0\right)$ & $n\in\mathbb{Z}$ & $class\left(1,0\right)$: $\mathscr{H}\left(1,0\right)$\tabularnewline
\hline 
$\left(n,n\right)$ & $n\in\mathbb{Z}$ & $class\left(1,1\right)$: $\mathscr{H}\left(1,1\right)$\tabularnewline
\hline 
$\left(n,2n\right)$ & $n\in\mathbb{Z}$ & $class\left(1,2\right)$: $\mathscr{H}\left(1,2\right)$\tabularnewline
\hline 
$\vdots$ &  & \tabularnewline
\hline 
$\left(n,kn\right)$ & $n\in\mathbb{Z}$ & $class\left(1,k\right)$: $\mathscr{H}\left(1,k\right)$\tabularnewline
\hline 
\end{tabular}
\par\end{center}
\begin{rem}
We may need the points in $\mathbb{Z}^{2}$ for an orthogonal in $L^{2}\left(J^{2}\right)$,
and we get orthogonal subspaces $\left\{ \mathscr{H}\left(1,k\right)\right\} _{k\in\mathbb{Z}}$,
orthogonality modulo the constants. But if $w=\left(1,\sqrt{2}\right)$
for example, then the $\mathbb{Z}^{2}$-representation is as follows
\[
e^{i\pi\left(x+\sqrt{2}y\right)}=e^{i\pi x}\sum_{n\in\mathbb{Z}}\frac{\sin\pi\sqrt{2}}{\pi\left(\sqrt{2}-n\right)}e^{i\pi ny}\in\sum_{n\in\mathbb{Z}}^{\oplus}\mathscr{H}\left(1,n\right).
\]
Moreover, 
\[
class\left(1,0\right)\cup class\left(0,1\right)\cup_{k\in\mathbb{Z}\backslash\left(0\right)}class\left(1,k\right)\cup class\left(\left(n_{1},n_{2}\right)\right)_{n_{1}\neq n_{2}}=\mathbb{Z}^{2};
\]
see \prettyref{exa:F3}.

All the classes intersect in $\left(0,0\right)$ correspond to the
index $c_{\left(0,0\right)}\mathbbm{1}$, where $\mathbbm{1}=e^{i\pi\left(0x+0y\right)}$
in the 2D Fourier expansion: 
\[
\text{Fix \ensuremath{k}:\ensuremath{\quad}}\sum_{n\in\mathbb{Z}}a_{n}e^{i\pi n\left(x+ky\right)}\sim\mathscr{H}\left(\left(1,k\right)\right).
\]
\end{rem}

\begin{question}
Display a complete list of points $w\in\mathbb{Z}^{2}\backslash\left(0\right)$
such that the corresponding subspaces $\mathscr{H}'\left(w\right):=\mathscr{H}\left(w\right)\ominus\mathbb{C}\mathbbm{1}$
are mutually orthogonal. 
\end{question}

\begin{defn}
We say that a point in $class\left(w\right)$ is rational iff $\exists$
$\left(q_{1},\cdots,q_{m}\right)$, $q_{i}\in\mathbb{Q}$, such that
$w\sim\left(q_{1},\cdots,q_{m}\right)$. In this case, we may pick
$\left(k_{1},\cdots,k_{m}\right)\in\mathbb{Z}^{m}$ such that $w\sim\left(k_{1},\cdots,k_{m}\right)$. 

A subset $W\subset\mathbb{R}^{m}\backslash\left(0\right)$ is said
to be rational iff each class contains a rational generator. We get
$\mathscr{H}_{W}=L^{2}\left(J^{m}\right)$ iff $W\subset P\left(\mathbb{R}^{m}\right)$
contains all the rational points. 
\end{defn}

\begin{defn}
A subset $W\subset\mathbb{Z}^{m}\backslash\left(0\right)$ is said
to be \emph{complete} iff (Def.)
\begin{align}
\bigcup_{w\in W}\mathbb{Z}w & =\mathbb{Z}^{m};\;\text{and}\\
\mathbb{Z}w\bigcap\mathbb{Z}w' & =0\;\text{when \ensuremath{w\neq w'.}}
\end{align}
\end{defn}

\begin{lem}
\label{lem:F11}If $W\subset\mathbb{Z}^{m}\backslash\left(0\right)$
is complete, then 
\begin{equation}
\sum_{w\in W}^{\oplus}\mathscr{H}\left(w\right)=L^{2}\left(J^{m}\right),\;\text{with}
\end{equation}
\begin{equation}
\mathscr{H}\left(w\right)\cap\mathscr{H}\left(w'\right)=\mathbb{C}\mathbbm{1}\;\text{when \ensuremath{w\neq w'}}.\label{eq:F20}
\end{equation}
And modulo constants, the subspaces are orthogonal. \emph{(In \eqref{eq:F20},
$\mathbb{C}\mathbbm{1}$ denotes }\textup{multiples of the constant
function $\mathbbm{1}$}.\emph{)}
\end{lem}

\begin{example}[$m=2$: Complete subsets in $\mathbb{Z}^{2}$]
\label{exa:F12} Let 
\[
W=\left\{ \left(1,0\right),\left(0,1\right),\left(n_{1},n_{2}\right)\right\} ,
\]
where $g.c.d\left(n_{1},n_{2}\right)=1$, and where g.c.d is short
for the \emph{greatest common divisor }(in $\mathbb{Z}_{+}$). See
Figures \ref{fig:F1}--\ref{fig:F2}.
\end{example}

\begin{figure}
\includegraphics[width=0.4\textwidth]{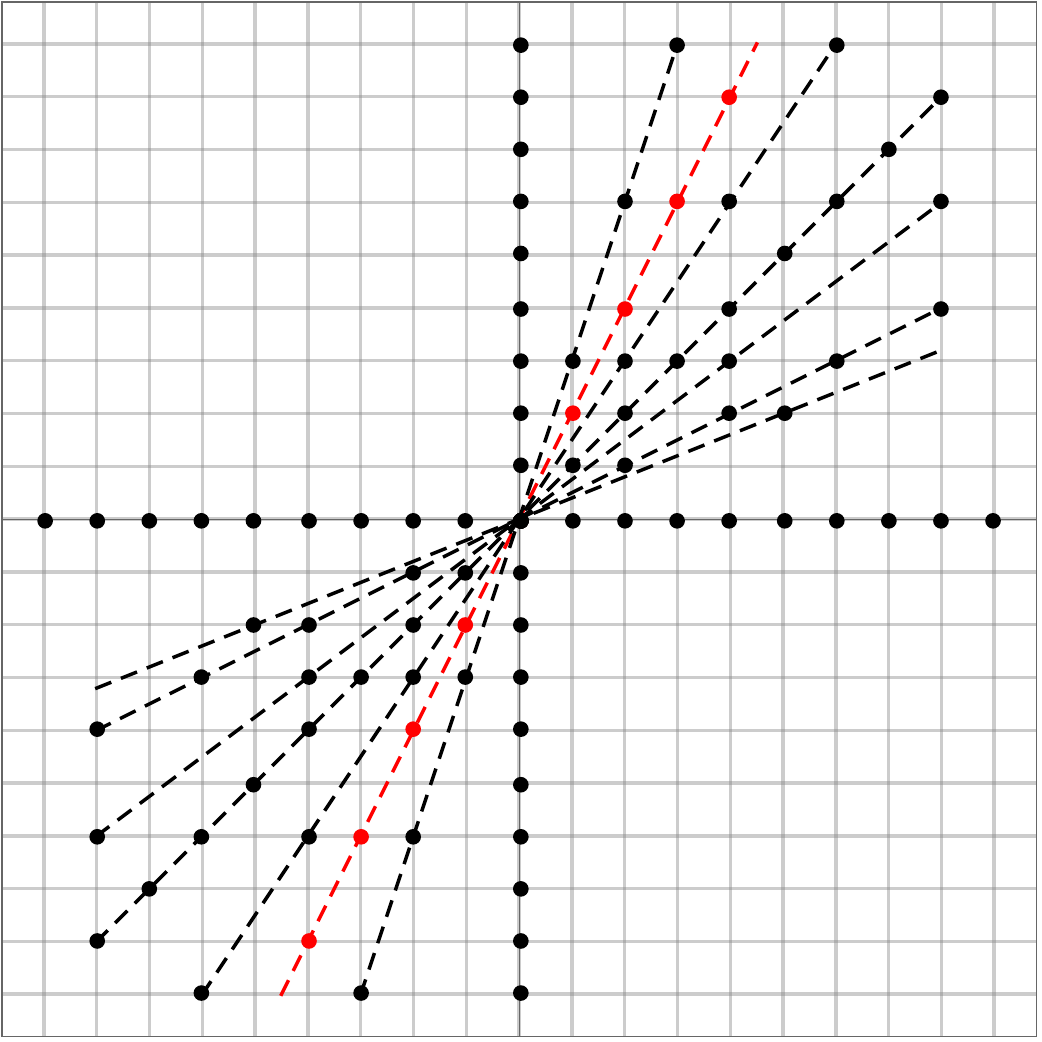}

\caption{\label{fig:F1}Part of a complete subset in $\mathbb{Z}^{2}$}

\end{figure}

\begin{figure}
\includegraphics[width=0.4\textwidth]{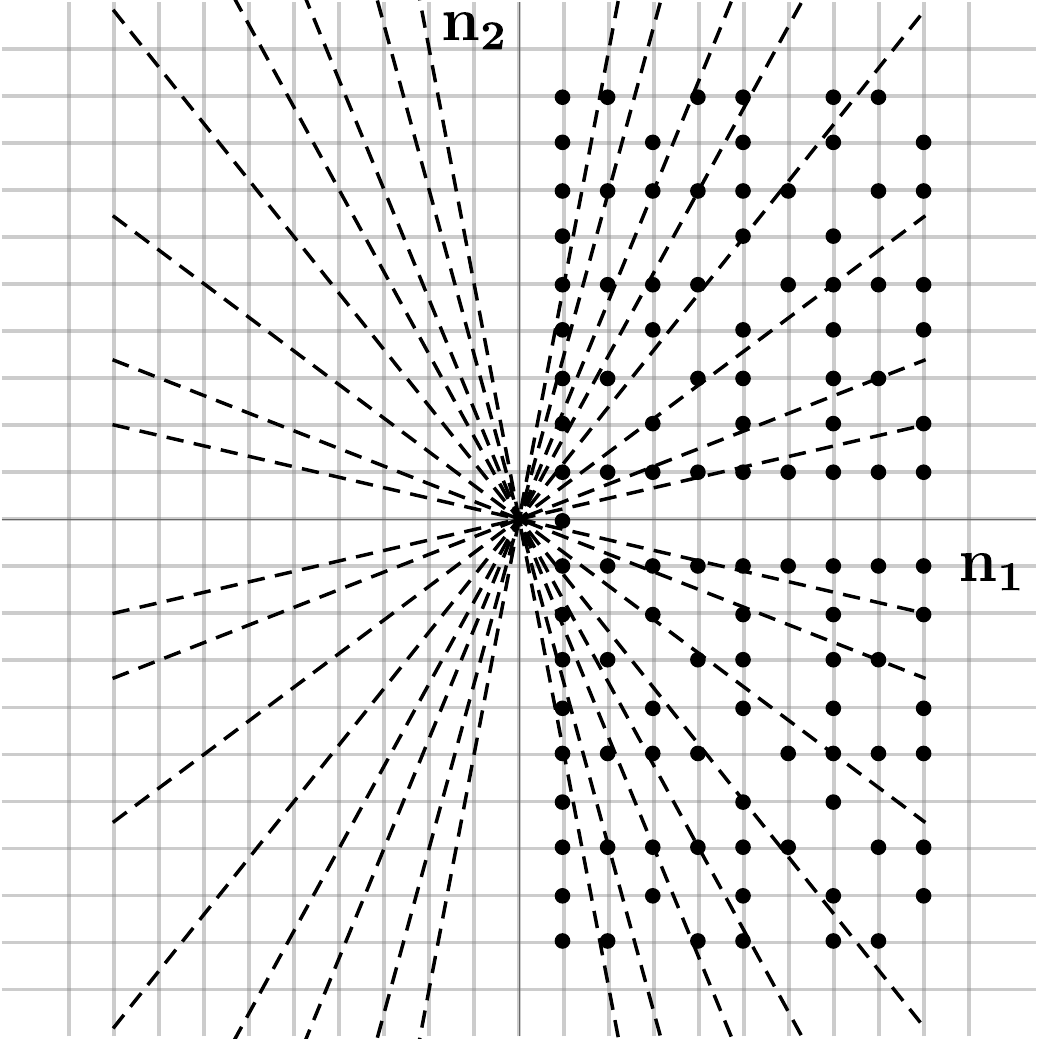}

\caption{\label{fig:F2}A subset of vectors $w$ in a set $W$ having the completeness
property from \prettyref{exa:F12}. Note that, for each discretized
line, we are specifying a $w$ yielding an irreducible direction;
one for each of the equivalence classes in $\mathbb{Z}^{2}$, as illustrated
in \prettyref{fig:F1} above. The idea is that, for a set $W$, we
pick only one vector $w$ for each of the discretized lines.}
\end{figure}

\begin{example}[$m>2$: Complete subsets in $\mathbb{Z}^{m}$]
\label{exa:F13}The union of the following subsets in $\mathbb{Z}^{m}\backslash\left(0\right)$: 
\begin{enumerate}
\item $\left(n_{1},n_{2},\cdots,n_{m}\right)\in\mathbb{Z}^{m}\backslash\left(0\right)$,
where $k$ out of the $m$ coordinates are $0$, and $k=1,2,3,\cdots,m-1$;
and g.c.d. for the non-zero coordinates = 1.
\item $\left(n_{1},n_{2},\cdots,n_{m}\right)\in\mathbb{Z}^{m}\backslash\left(0\right)$,
where all $n_{j}\neq0$, $1\leq j\leq m$, and $g.c.d\left\{ n_{j}\right\} _{j=1}^{m}=1$. 
\end{enumerate}
\end{example}

For additional details regarding reproducing kernel Hilbert spaces,
see e.g., \cite{MR3782384,MR3795680,MR3806191,MR0051437}.

\section{\label{sec:fr}Fourier representation}

The purpose of this section is to make precise a certain Fourier/harmonic
analysis representation for the UAT.

It suffices to take $\varphi\in C\left(\mathbb{R},\mathbb{C}\right)$
to be $2$-periodic, i.e., $\varphi\left(s+2n\right)=\varphi\left(s\right)$,
$\forall s\in\mathbb{R}$, $n\in\mathbb{Z}$. We then have the usual
Fourier expansion 
\begin{align*}
\varphi\left(s\right) & =\sum_{n\in\mathbb{Z}}\widehat{\varphi}\left(n\right)e^{i\pi ns},\\
\widehat{\varphi}\left(n\right) & =\frac{1}{2}\int_{-1}^{1}\varphi\left(s\right)e^{-i\pi ns}ds,
\end{align*}
and 
\[
\frac{1}{2}\int_{-1}^{1}\left|\varphi\left(s\right)\right|^{2}ds=\sum_{n\in\mathbb{Z}}\left|\widehat{\varphi}\left(n\right)\right|^{2}.
\]
Now fix $w\in\mathbb{Z}^{m}\backslash\left(0\right)$. Then
\[
F\left(x\right):=\varphi\left(w^{T}x\right)=\varphi\left(w\cdot x\right)
\]
has the representation 
\[
F\left(x\right)=\sum_{k\in\mathbb{Z}}\widehat{\varphi}\left(k\right)e^{i\pi kw^{T}x}=\sum_{k\in\mathbb{Z}}\widehat{\varphi}\left(k\right)e^{i\pi\left(kw\right)\cdot x};
\]
and 
\[
\left\Vert F\right\Vert _{L^{2}\left(J^{m}\right)}^{2}=\frac{1}{2^{m}}\int_{J^{m}}\left|F\left(x\right)\right|^{2}d^{m}x=\sum_{k\in\mathbb{Z}}\left|\widehat{\varphi}\left(k\right)\right|^{2}.
\]

\begin{rem}
~
\begin{enumerate}
\item We consider annihilation measures $\mu\in W_{\mathscr{M}}^{\perp}$,
but these measures must necessarily be singular, albeit of finite
total variation.
\item Our reasoning here extends the argument given in \prettyref{exa:wavelet}
below. The setting in the example is specialized here in order to
highlight the general idea.
\end{enumerate}
\begin{enumerate}[resume]
\item We begin with Parseval in one dimension as follows:
\[
\int_{-\infty}^{\infty}\left|\psi\right|^{2}dx=\frac{1}{2\pi}\int_{-\infty}^{\infty}\left|\widehat{\psi}\left(\xi\right)\right|^{2}d\xi=2.
\]
Fix $W\subset\mathbb{R}^{m}\backslash\left(0\right)$. Recall $\mu\in W_{\mathscr{M}}^{\perp}\Longleftrightarrow$
$\mu_{w}=0$ $\forall w\in W$ $\Longleftrightarrow$ $\widehat{\mu}_{w}\equiv0$
$\forall w\in W$ $\Longleftrightarrow$ $\widehat{\mu}\left(tw\right)=0$
$\forall t\in\mathbb{R}$, $\forall w\in W$. And one may apply this
to $d\mu\left(x\right)=F\left(x\right)d^{m}x$, $F\in L^{2}\left(J^{m}\right)$.
\end{enumerate}
\end{rem}

\begin{example}[Some wavelet functions, $m=2$]
\label{exa:wavelet}Set $\psi=$ Haar wavelet function on $\mathbb{R}$,
up to normalization, so that 
\[
\widehat{\psi}\left(\xi\right)=\frac{\cos\left(\xi\right)-1}{\xi},\quad\xi\in\mathbb{R};
\]
and $\widehat{\psi}\left(0\right)=0$. Let 
\begin{align*}
F_{1}\left(x_{1},x_{2}\right) & =f\left(x_{1}\right)\psi\left(x_{2}\right),\\
F_{2}\left(x_{1},x_{2}\right) & =\psi\left(x_{1}\right)g\left(x_{2}\right),
\end{align*}
and $f,g$ are arbitrary. Then $F=F_{1}\ast F_{2}\in W^{\perp}$,
where $W=\left\{ \left(\begin{smallmatrix}1\\
0
\end{smallmatrix}\right),\left(\begin{smallmatrix}0\\
1
\end{smallmatrix}\right)\right\} $. 

Similarly, for $m=3$, set 
\begin{align*}
F_{1}\left(x_{1},x_{2},x_{3}\right) & =f\left(x_{1},x_{2}\right)\psi\left(x_{3}\right),\\
F_{2}\left(x_{1},x_{2},x_{3}\right) & =g\left(x_{1},x_{3}\right)\psi\left(x_{2}\right),\\
F_{3}\left(x_{1},x_{2},x_{3}\right) & =h\left(x_{2},x_{3}\right)\psi\left(x_{1}\right),
\end{align*}
$f,g,h$ arbitrary. Then 
\[
F=F_{1}\ast F_{2}\ast F_{3}\in W^{\perp},
\]
where $W=\left\{ \left(\begin{smallmatrix}1\\
0\\
0
\end{smallmatrix}\right),\left(\begin{smallmatrix}0\\
1\\
0
\end{smallmatrix}\right),\left(\begin{smallmatrix}0\\
0\\
1
\end{smallmatrix}\right)\right\} $. 

The first two terms in the Taylor expansion of $\widehat{\psi}\left(\xi\right)$
is 
\[
\widehat{\psi}\left(\xi\right)=-\frac{1}{2}\xi+\frac{1}{24}\xi^{3}-\cdots
\]

In the remaining of the section, we discuss choices of sets $W$ of
admissible directions to be used in our transform analysis. As well
as some general properties for these sets.
\end{example}

\begin{cor}
$W^{\perp}$ is infinite-dimensional.
\end{cor}

Now fix $w\in W$, and do the $\psi$ construction in a concatenate
system $\mathbb{R}w\times\Pi_{w}=\mathbb{R}^{m}$, where $\Pi_{w}=\left\{ x\in\mathbb{R}^{m}\mathrel{;}w^{T}x=0\right\} $
= the $w$ hyperplane. 

For $z\in\Pi_{w}\backslash\left(0\right)$, do a $\psi$ construction
and extend to $\mathbb{R}^{m}$, so 
\begin{equation}
\widehat{\psi_{z}}\left(\xi\right)=\frac{\cos\left(z^{T}\xi\right)-1}{z^{T}\xi}=\frac{\cos\left(z\cdot\xi\right)-1}{z\cdot\xi},\quad\forall\xi\in\mathbb{R}^{m}.\label{eq:G1}
\end{equation}
But we should cut down $\psi_{z}$ to $J^{m}$, so that integration
is convergent. 
\begin{cor}
Let $\psi_{z}$ be as above, and let $w$ be fixed; s.t. $w\in\Pi_{z}$,
$z\in\Pi_{w}$. Then 
\begin{equation}
\widehat{\psi}_{z}\left(tw\right)=0,\quad\forall t\in\mathbb{R},\label{eq:G2}
\end{equation}
and in particular, $\psi_{z}\in W_{2}^{\perp}$. 
\end{cor}

\begin{proof}
Observe that 
\begin{equation}
\widehat{\psi}_{z}\left(tw\right)=\frac{\cos\left(tw^{T}z\right)-1}{tw^{T}z}=0
\end{equation}
since $w^{T}z=0$, see \eqref{eq:G1}; so $\psi_{z}\in W_{2}^{\perp}\subset W_{\mathscr{M}}^{\perp}$
(but it depends on $w$, fixed in $W$.) If $W=\left\{ w_{k}\right\} _{1}^{p}$,
and 
\begin{equation}
\psi_{k}\in W_{2}^{\perp},\label{eq:G4}
\end{equation}
then 
\begin{equation}
F=\ast_{k=1}^{p}\psi_{k}\in W_{2}^{\perp}\label{eq:G5}
\end{equation}
where $\ast$ denotes convolution. Note functions are restricted to
$J^{m}$. 
\end{proof}
Here is a way to generate more functions $\psi_{k}$ as in \eqref{eq:G4}--\eqref{eq:G5}:
We can easily generalize to more functions $F\in W_{2}^{\perp}$. 

Fix $W=\left\{ w_{k}\right\} _{1}^{p}$. Let $\mu_{w_{1}}$ be a
finite positive measure on $\Pi_{w_{1}}$. Let $\psi_{z}$ be as in
\eqref{eq:G1}--\eqref{eq:G2}, $z\in\Pi_{w_{1}}$, and set 
\begin{equation}
\psi_{w_{1}}\left(\cdot\right)=\int_{\Pi_{w_{1}}}\psi_{z}\left(\cdot\right)d\mu_{w_{1}}\left(z\right)\label{eq:G6}
\end{equation}
as a function on $\mathbb{R}^{m}$, and restrict to $J^{m}$. Now
do the construction in \eqref{eq:G6} also for $w_{2}$, $w_{3}$,
$\cdots$ , $w_{p}$, with choice of positive measures $\mu_{w_{k}}$
on $\Pi_{w_{k}}$ for $1\leq k\leq p$; and set 
\begin{equation}
\psi=\ast_{k=1}^{p}\psi_{k},\label{eq:G7}
\end{equation}
so that $\widehat{\psi}=\prod_{k=1}^{p}\widehat{\psi}_{w_{k}}$.
We conclude that $\psi$ in \eqref{eq:G7} is in $W_{2}^{\perp}$.
\begin{example}
Illustration of key arguments in one and two dimensions.

For $m=2$, $x=\left(x_{1},x_{2}\right)\in\mathbb{R}^{2}$ or $x\in J^{2}$;
let $w=e_{2}$ and
\[
\psi_{w_{2}}\left(x_{1},x_{2}\right)=\psi\left(x_{1}\right)
\]
where $\psi$ is the 1D Haar wavelet. Then 
\begin{align*}
\widehat{\psi_{w_{2}}}\left(\xi_{1},\xi_{2}\right) & =\int_{-1}^{1}\int_{-1}^{1}e^{i\left(x_{1}\xi_{1}+x_{2}\xi_{2}\right)}\psi\left(x_{1}\right)dx_{1}dx_{2}\\
 & =\frac{\left(\cos\left(\xi_{1}\right)-1\right)}{\xi_{1}}\frac{\sin\left(\xi_{2}\right)}{\xi_{2}};
\end{align*}
recall $\text{sinc}\left(\xi\right)=\frac{\sin\xi}{\xi}$. So we have
\[
\widehat{\psi_{w_{1}}\ast\psi_{w_{2}}}\left(\xi_{1},\xi_{2}\right)=\frac{\cos\left(\xi_{2}\right)-1}{\xi_{2}}\text{sinc}\left(\xi_{1}\right)\frac{\cos\left(\xi_{1}\right)-1}{\xi_{1}}\text{sinc}\left(\xi_{2}\right)
\]
and so $F=\psi_{w_{1}}\ast\psi_{w_{2}}\in W^{\perp}$, where $W=\left\{ \left(0,1\right),\left(1,0\right)\right\} $,
$\widehat{F}\left(tw_{1}\right)=0$ and $\widehat{F}\left(tw_{2}\right)=0$,
$\forall t\in\mathbb{R}$. 
\end{example}

$m$-dimension ($m>2$): Fix $w\in W$, consider
\[
\int_{\prod_{w}}\widehat{\psi_{z}}\left(tw\right)d\mu_{w}\left(z\right).
\]
Set $F_{w}\left(x\right)=\int_{\Pi_{w}}\psi_{z}\left(x\right)d\mu_{w}\left(z\right)$,
then 
\[
\int_{J^{m}}\varphi\left(w^{T}x\right)F_{w}\left(x\right)d^{m}x
\]

For additional details regarding Wiener theory and positive definite
functions, see e.g., \cite{MR2948566,MR3830229}.
\begin{acknowledgement*}
The present work was started during the NSF CBMS Conference, \textquotedblleft Harmonic
Analysis: Smooth and Non-Smooth\textquotedblright , by Jorgensen,
held at the Iowa State University, June 4--8, 2018. We thank the
NSF for funding, the organizers, especially Prof Eric Weber; and the
CBMS participants. We had many fruitful discussions with Profs Daniel
Alpay, and Sergii Bezuglyi, among others. 
\end{acknowledgement*}
\bibliographystyle{IEEEtran}
\bibliography{ref}

\end{document}